\documentclass[11pt]{amsart}
\usepackage{amsthm}
\usepackage{amsmath}
\usepackage{amssymb}
\usepackage{verbatim}
\usepackage{mathrsfs}
\usepackage{eucal}
\let\mathcal=\CMcal
\usepackage[dvipdfm,bookmarks=true,bookmarksnumbered=true]{hyperref}

\allowdisplaybreaks[4]

\begin{document}
\def\sbt{\raisebox{1.2pt}{$\scriptscriptstyle\,\bullet\,$}}

\def\alp{\alpha}
\def\bet{\beta}
\def\gam{\gamma}
\def\del{\delta}
\def\eps{\epsilon}
\def\zet{\zeta}
\def\tht{\theta}
\def\iot{\iota}
\def\kap{\kappa}
\def\lam{\lambda}
\def\sig{\sigma}
\def\ome{\omega}
\def\vep{\varepsilon}
\def\vth{\vartheta}
\def\vpi{\varpi}
\def\vrh{\varrho}
\def\vsi{\varsigma}
\def\vph{\varphi}
\def\Gam{\Gamma}
\def\Del{\Delta}
\def\Tht{\Theta}
\def\Lam{\Lambda}
\def\Sig{\Sigma}
\def\Ups{\Upsilon}
\def\Ome{\Omega}
\def\vka{\varkappa}
\def\vDe{\varDelta}
\def\vSi{\varSigma}
\def\vTh{\varTheta}
\def\vGm{\varGamma}
\def\vOm{\varOmega}
\def\vPi{\varPi}
\def\vPh{\varPhi}
\def\vPs{\varPsi}
\def\vUp{\varUpsilon}
\def\vXi{\varXi}

\def\frka{{\mathfrak a}}    \def\frkA{{\mathfrak A}}
\def\frkb{{\mathfrak b}}    \def\frkB{{\mathfrak B}}
\def\frkc{{\mathfrak c}}    \def\frkC{{\mathfrak C}}
\def\frkd{{\mathfrak d}}    \def\frkD{{\mathfrak D}}
\def\frke{{\mathfrak e}}    \def\frkE{{\mathfrak E}}
\def\frkf{{\mathfrak f}}    \def\frkF{{\mathfrak F}}
\def\frkg{{\mathfrak g}}    \def\frkG{{\mathfrak G}}
\def\frkh{{\mathfrak h}}    \def\frkH{{\mathfrak H}}
\def\frki{{\mathfrak i}}    \def\frkI{{\mathfrak I}}
\def\frkj{{\mathfrak j}}    \def\frkJ{{\mathfrak J}}
\def\frkk{{\mathfrak k}}    \def\frkK{{\mathfrak K}}
\def\frkl{{\mathfrak l}}    \def\frkL{{\mathfrak L}}
\def\frkm{{\mathfrak m}}    \def\frkM{{\mathfrak M}}
\def\frkn{{\mathfrak n}}    \def\frkN{{\mathfrak N}}
\def\frko{{\mathfrak o}}    \def\frkO{{\mathfrak O}}
\def\frkp{{\mathfrak p}}    \def\frkP{{\mathfrak P}}
\def\frkq{{\mathfrak q}}    \def\frkQ{{\mathfrak Q}}
\def\frkr{{\mathfrak r}}    \def\frkR{{\mathfrak R}}
\def\frks{{\mathfrak s}}    \def\frkS{{\mathfrak S}}
\def\frkt{{\mathfrak t}}    \def\frkT{{\mathfrak T}}
\def\frku{{\mathfrak u}}    \def\frkU{{\mathfrak U}}
\def\frkv{{\mathfrak v}}    \def\frkV{{\mathfrak V}}
\def\frkw{{\mathfrak w}}    \def\frkW{{\mathfrak W}}
\def\frkx{{\mathfrak x}}    \def\frkX{{\mathfrak X}}
\def\frky{{\mathfrak y}}    \def\frkY{{\mathfrak Y}}
\def\frkz{{\mathfrak z}}    \def\frkZ{{\mathfrak Z}}

\def\cal{\fam2}
\def\cala{{\cal A}}
\def\calb{{\cal B}}
\def\calc{{\cal C}}
\def\cald{{\cal D}}
\def\cale{{\cal E}}
\def\calf{{\cal F}}
\def\calg{{\cal G}}
\def\calh{{\cal H}}
\def\cali{{\cal I}}
\def\calj{{\cal J}}
\def\calk{{\cal K}}
\def\call{{\cal L}}
\def\calm{{\cal M}}
\def\caln{{\cal N}}
\def\calo{{\cal O}}
\def\calp{{\cal P}}
\def\calq{{\cal Q}}
\def\calr{{\cal R}}
\def\cals{{\cal S}}
\def\calt{{\cal T}}
\def\calu{{\cal U}}
\def\calv{{\cal V}}
\def\calw{{\cal W}}
\def\calx{{\cal X}}
\def\caly{{\cal Y}}
\def\calz{{\cal Z}}

\def\AA{{\mathbb A}}
\def\BB{{\mathbb B}}
\def\CC{{\mathbb C}}
\def\DD{{\mathbb D}}
\def\EE{{\mathbb E}}
\def\FF{{\mathbb F}}
\def\GG{{\mathbb G}}
\def\HH{{\mathbb H}}
\def\II{{\mathbb I}}
\def\JJ{{\mathbb J}}
\def\KK{{\mathbb K}}
\def\LL{{\mathbb L}}
\def\MM{{\mathbb M}}
\def\NN{{\mathbb N}}
\def\OO{{\mathbb O}}
\def\PP{{\mathbb P}}
\def\QQ{{\mathbb Q}}
\def\RR{{\mathbb R}}
\def\SS{{\mathbb S}}
\def\TT{{\mathbb T}}
\def\UU{{\mathbb U}}
\def\VV{{\mathbb V}}
\def\WW{{\mathbb W}}
\def\XX{{\mathbb X}}
\def\YY{{\mathbb Y}}
\def\ZZ{{\mathbb Z}}

\def\bfa{{\mathbf a}}    \def\bfA{{\mathbf A}}
\def\bfb{{\mathbf b}}    \def\bfB{{\mathbf B}}
\def\bfc{{\mathbf c}}    \def\bfC{{\mathbf C}}
\def\bfd{{\mathbf d}}    \def\bfD{{\mathbf D}}
\def\bfe{{\mathbf e}}    \def\bfE{{\mathbf E}}
\def\bff{{\mathbf f}}    \def\bfF{{\mathbf F}}
\def\bfg{{\mathbf g}}    \def\bfG{{\mathbf G}}
\def\bfh{{\mathbf h}}    \def\bfH{{\mathbf H}}
\def\bfi{{\mathbf i}}    \def\bfI{{\mathbf I}}
\def\bfj{{\mathbf j}}    \def\bfJ{{\mathbf J}}
\def\bfk{{\mathbf k}}    \def\bfK{{\mathbf K}}
\def\bfl{{\mathbf l}}    \def\bfL{{\mathbf L}}
\def\bfm{{\mathbf m}}    \def\bfM{{\mathbf M}}
\def\bfn{{\mathbf n}}    \def\bfN{{\mathbf N}}
\def\bfo{{\mathbf o}}    \def\bfO{{\mathbf O}}
\def\bfp{{\mathbf p}}    \def\bfP{{\mathbf P}}
\def\bfq{{\mathbf q}}    \def\bfQ{{\mathbf Q}}
\def\bfr{{\mathbf r}}    \def\bfR{{\mathbf R}}
\def\bfs{{\mathbf s}}    \def\bfS{{\mathbf S}}
\def\bft{{\mathbf t}}    \def\bfT{{\mathbf T}}
\def\bfu{{\mathbf u}}    \def\bfU{{\mathbf U}}
\def\bfv{{\mathbf v}}    \def\bfV{{\mathbf V}}
\def\bfw{{\mathbf w}}    \def\bfW{{\mathbf W}}
\def\bfx{{\mathbf x}}    \def\bfX{{\mathbf X}}
\def\bfy{{\mathbf y}}    \def\bfY{{\mathbf Y}}
\def\bfz{{\mathbf z}}    \def\bfZ{{\mathbf Z}}

\def\scra{{\mathscr A}}
\def\scrb{{\mathscr B}}
\def\scrc{{\mathscr C}}
\def\scrd{{\mathscr D}}
\def\scre{{\mathscr E}}
\def\scrf{{\mathscr F}}
\def\scrg{{\mathscr G}}
\def\scrh{{\mathscr H}}
\def\scri{{\mathscr I}}
\def\scrj{{\mathscr J}}
\def\scrk{{\mathscr K}}
\def\scrl{{\mathscr L}}
\def\scrm{{\mathscr M}}
\def\scrn{{\mathscr N}}
\def\scro{{\mathscr O}}
\def\scrp{{\mathscr P}}
\def\scrq{{\mathscr Q}}
\def\scrr{{\mathscr R}}
\def\scrs{{\mathscr S}}
\def\scrt{{\mathscr T}}
\def\scru{{\mathscr U}}
\def\scrv{{\mathscr V}}
\def\scrw{{\mathscr W}}
\def\scrx{{\mathscr X}}
\def\scry{{\mathscr Y}}
\def\scrz{{\mathscr Z}}

\def\phm{\phantom}
\def\smallstrut{\vphantom{\vrule height 3pt }}
\def\bdm #1#2#3#4{\left(
\begin{array} {c|c}{\ds{#1}}
 & {\ds{#2}} \\ \hline
{\ds{#3}\vphantom{\ds{#3}^1}} &  {\ds{#4}}
\end{array}
\right)}

\def\GL{\mathrm{GL}}
\def\PGL{\mathrm{PGL}}
\def\SL{\mathrm{SL}}
\def\Mp{\mathrm{Mp}}
\def\PGSp{\mathrm{PGSp}}
\def\SK{\mathrm{SK}}
\def\Ik{\mathrm{Ik}}
\def\SU{\mathrm{SU}}
\def\SO{\mathrm{SO}}
\def\GSO{\mathrm{GSO}}
\def\GSp{\mathrm{GSp}}
\def\GO{\mathrm{GO}}
\def\GU{\mathrm{GU}}
\def\PGO{\mathrm{PGO}}
\def\U{\mathrm{U}}
\def\O{\mathrm{O}}
\def\G{\mathrm{G}}
\def\Mat{\mathrm{M}}
\def\Tr{\mathrm{Tr}}
\def\tr{\mathrm{tr}}
\def\Ad{\mathrm{Ad}}
\def\St{\mathrm{St}}
\def\new{\mathrm{new}}
\def\Wh{\mathrm{Wh}}
\def\FJ{\mathrm{FJ}}
\def\Fj{\mathrm{Fj}}
\def\Sym{\mathrm{Sym}}
\def\Her{\mathrm{Her}}
\def\sym{\mathrm{sym}}
\def\supp{\mathrm{supp}}
\def\proj{\mathrm{proj}}
\def\Hom{\mathrm{Hom}}
\def\End{\mathrm{End}}
\def\Ker{\mathrm{Ker}}
\def\Res{\mathrm{Res}}
\def\res{\mathrm{res}}
\def\cusp{\mathrm{cusp}}
\def\Irr{\mathrm{Irr}}
\def\rank{\mathrm{rank}}
\def\sgn{\mathrm{sgn}}
\def\diag{\mathrm{diag}}
\def\Wd{\mathrm{Wd}}
\def\nd{\mathrm{nd}}
\def\d{\mathrm{d}}
\def\rmm{\mathrm{m}}
\def\rmMp{\mathrm{Mp}}
\def\rmSp{Sp}
\def\rmO{\mathrm{O}}
\def\rmSO{\mathrm{SO}}
\def\rmGS{\mathrm{GS}}
\def\rmGO{\mathrm{GO}}
\def\rmGU{\mathrm{GU}}
\def\La{\langle}
\def\Ra{\rangle}

\def\trs{\,^t\!}
\def\tri{\,^\iot\!}
\def\iu{\sqrt{-1}}
\def\oo{\hbox{\bf 0}}
\def\ono{\hbox{\bf 1}}
\def\smallcirc{\lower .3em \hbox{\rm\char'27}\!}
\def\AAf{\AA_\bff}
\def\thalf{\tfrac{1}{2}}
\def\bsl{\backslash}
\def\wtl{\widetilde}
\def\til{\tilde}
\def\Ind{\operatorname{Ind}}
\def\ind{\operatorname{ind}}
\def\cind{\operatorname{c-ind}}
\def\ord{\operatorname{ord}}
\def\beq{\begin{equation}}
\def\eeq{\end{equation}}
\def\d{\mathrm{d}}

\newcounter{one}
\setcounter{one}{1}
\newcounter{two}
\setcounter{two}{2}
\newcounter{thr}
\setcounter{thr}{3}
\newcounter{fou}
\setcounter{fou}{4}
\newcounter{fiv}
\setcounter{fiv}{5}
\newcounter{six}
\setcounter{six}{6}
\newcounter{sev}
\setcounter{sev}{7}
\newcounter{eig}
\setcounter{eig}{8}
\newcounter{nin}
\setcounter{nin}{9}

\newcommand{\shp}{\rm\char'43}

\def\lddots{\mathinner{\mskip1mu\raise1pt\vbox{\kern7pt\hbox{.}}\mskip2mu\raise4pt\hbox{.}\mskip2mu\raise7pt\hbox{.}\mskip1mu}}
\newcommand{\1}{1\hspace{-0.25em}{\rm{l}}}

\makeatletter
\def\varddots{\mathinner{\mkern1mu
    \raise\p@\hbox{.}\mkern2mu\raise4\p@\hbox{.}\mkern2mu
    \raise7\p@\vbox{\kern7\p@\hbox{.}}\mkern1mu}}
\makeatother

\def\today{\ifcase\month\or
 January\or February\or March\or April\or May\or June\or
 July\or August\or September\or October\or November\or December\fi
 \space\number\day, \number\year}

\makeatletter
\def\varddots{\mathinner{\mkern1mu
    \raise\p@\hbox{.}\mkern2mu\raise4\p@\hbox{.}\mkern2mu
    \raise7\p@\vbox{\kern7\p@\hbox{.}}\mkern1mu}}
\makeatother

\def\today{\ifcase\month\or
 January\or February\or March\or April\or May\or June\or
 July\or August\or September\or October\or November\or December\fi
 \space\number\day, \number\year}

\makeatletter
\@addtoreset{equation}{section}
\def\theequation{\thesection.\arabic{equation}}

\theoremstyle{plain}
\newtheorem{theorem}{Theorem}[section]
\newtheorem*{theorem_a}{Theorem \ref{thm:52}}
\newtheorem*{theorem_b}{Theorem \ref{thm:51}}
\newtheorem*{theorem_c}{Theorem \ref{thm:91}}
\newtheorem*{theorem_d}{Theorem \ref{thm:92}}
\newtheorem*{corollary_a}{Corollary A}
\newtheorem{lemma}[theorem]{Lemma}
\newtheorem{proposition}[theorem]{Proposition}
\theoremstyle{definition}
\newtheorem{definition}[theorem]{Definition}
\newtheorem{conjecture}[theorem]{Conjecture}
\theoremstyle{remark}
\newtheorem{remark}[theorem]{Remark}
\newtheorem*{main_remark}{Remark}
\newtheorem{corollary}[theorem]{Corollary}

\renewcommand{\thepart}{\Roman{part}}
\setcounter{tocdepth}{1}
\setcounter{section}{0} 


\title[]{\bf The CAP representations indexed by Hilbert cusp forms} 
\author{Shunsuke Yamana}
\address{Department of mathematics, Kyoto University, Kitashirakawa Oiwake-cho, Sakyo-ku, Kyoto, 606-8502, Japan/Hakubi Center, Yoshida-Ushinomiya-cho, Sakyo-ku, Kyoto, 606-8501, Japan}
\email{yamana07@math.kyoto-u.ac.jp}
\begin{abstract} 
We combine the Duke-Imamoglu-Ikeda lifting with the theta lifting to produce new CAP representations of metaplectic, symplectic and orthogonal groups. 
These constructions partially generalize the theories of Waldspurger on the Shimura correspondence and of Piatetski-Shapiro on the Saito-Kurokawa lifting to higher dimensions. 
Applications include a relation between Fourier coefficients of Hilbert cusp forms of weight $k+\frac{1}{2}$ and a weighted sum of the representation numbers of a quadratic form of rank $2k$ by a quadratic form of rank $4k$. 
\end{abstract}
\maketitle

\tableofcontents

\section{Introduction}\label{sec:1}

Let $F$ be a totally real number field of degree $d$ with ad\`{e}le ring $\AA$. 
We denote the set of archimedean primes of $F$ by $\frkS_\infty$ and the normalized absolute value by $\alp=\prod_v\alp_v:\AA^\times\to\RR^\times_+$. 
Define a character $\bfe:\RR\to\SS$ by $\bfe(x)=e^{2\pi\iu x}$.  
Let $\psi=\prod_v\psi_v$ be the additive character of $\AA/F$ such that $\psi_v=\bfe$ for $v\in\frkS_\infty$. 
For $\eta\in F^\times_v$ we define a character $\psi^\eta_v$ of $F_v$ by $\psi^\eta_v(x)=\psi_v(\eta x)$ and have an $8^\mathrm{th}$ root of unity $\gam(\psi_v^\eta)$ such that for all Schwartz functions $\phi$ on $F_v$ 
\[\int_{F_v}\phi(x_v)\psi_v(\eta x_v^2)\,\d x_v=\gam(\psi_v^\eta)\alp_v(2\eta)^{-1/2}\int_{F_v}\calf\phi(x_v)\psi_v\left(-\frac{x_v^2}{4\eta}\right)\d x_v, \]
where $\d x_v$ is the self-dual Haar measure on $F_v$ with respect to the Fourier transform $\calf\phi(y)=\int_{F_v}\phi(x_v)\psi_v(x_vy)\,\d x_v$. 
Set $\gam^{\psi_v}(\eta)=\gam(\psi_v)/\gam(\psi^\eta_v)$. 

We denote by $\Mp(W_n)_\AA$ the metaplectic double cover of the symplectic group $Sp_n(\AA)$ of rank $n$ and by $\scra_\cusp(\Mp(W_n))$ the space of cusp forms on $Sp_n(F)\bsl\Mp(W_n)$. 
Denote the inverse image of $Sp_n(F_v)$ by $\Mp(W_n)_v$. 
Given an irreducible infinite dimensional unitary representation $\pi_v$ of $\PGL_2(F_v)$, Waldspurger \cite{W2} associated a packet $\rmMp_1^{\psi_v}(\pi_v)$ of irreducible genuine unitary representations of $\Mp(W_1)_v$. 
This packet has cardinality two or one, depending on whether $\pi_v$ is a discrete series representation or not. 
Thus it can be written as $\rmMp_1^{\psi_v}(\pi_v)=\{\rmMp_1^{\psi_v}(\pi_v^+),\;\rmMp_1^{\psi_v}(\pi_v^-)\}$, where $\rmMp_1^{\psi_v}(\pi_v^-)$ is interpreted as $0$ if $\pi_v$ is not a discrete series representation. 

The packet $\rmMp_1^{\psi_v}(\pi_v)$ is constructed by using the Howe correspondence (associated to $\psi_v$) furnished by the dual pairs:

\begin{picture}(300,45)
\put(105,33){$\PGL_2(F_v)$}
\put(120,0){$\mathrm{P}D_v^\times$}
\put(225,16){$\Mp(W_1)_v$}
\put(158,36){\vector(4,-1){60}}
\put(158,3){\vector(4,1){60}}
\put(130,10){\vector(0,1){18}}
\put(40,16){\small Jacquet-Langlands}
\put(190,32){\small Howe}
\put(190,1){\small Howe}
\end{picture}

\noindent Here $D_v$ denotes the unique quaternion division algebra over $F_v$. 
The main point about this diagram is that it does not commute. 
This is an instance of a general feature of correspondences associated to dual pairs. 
If we set $\pi_v^+=\pi_v$ and write $\pi^-_v$ for the Jacquet-Langlands lift of $\pi_v$ to $\mathrm{P}D_v^\times$, then $\rmMp_1^{\psi_v}(\pi_v^\pm)$ is the theta lift of $\pi_v^\pm$ to $\Mp(W_1)_v$. 
 
In the first half of this paper we take $n$ to be odd.  
Let $\scrp_n$ denote a parabolic subgroup of $Sp_n$ with Levi subgroup $\GL_2\times\cdots\times\GL_2\times\SL_2$. 
The letter $\frkp$ denotes a nonarchimedean prime of $F$. 
We now define the local $A$-packet of irreducible genuine unitary representations of $\Mp(W_n)_\frkp$ by 
\[\rmMp_n^{\psi_\frkp}(\pi_\frkp)=\{\rmMp_n^{\psi_\frkp}(\pi_\frkp^+),\;\rmMp_n^{\psi_\frkp}(\pi_\frkp^-)\}, \]
where $\rmMp_n^{\psi_\frkp}(\pi_\frkp^\pm)$ is the unique irreducible quotient of the standard module 
\[\Ind^{\Mp(W_n)_\frkp}_{\til\scrp_n}(\pi_\frkp\otimes\gam^{\psi_\frkp}\alp_\frkp^{(n-1)/2})\boxtimes\cdots\boxtimes(\pi_\frkp\otimes\gam^{\psi_\frkp}\alp^2_\frkp)\boxtimes(\pi_\frkp\otimes\gam^{\psi_\frkp}\alp_\frkp)\boxtimes \rmMp_1^{\psi_\frkp}(\pi_\frkp^\pm). \]

For $k\in\NN$ we denote the discrete series representation of $\PGL_2(\RR)$ with extremal weights $\pm 2k$ by $D_{2k}$. 
Let $\frkD_\ell^{(n)}$ (resp. $\bar\frkD_\ell^{(n)}$) denote the lowest (resp. highest) weight representation of the real metaplectic group of rank $n$ with lowest (resp. highest) $K$-type $(\det)^\ell$ (resp. $(\det)^{-\ell}$) for a positive half-integer $\ell$. 
Define 
\[\rmMp_n^\bfe(D^{}_{2k})=\{\rmMp_n^\bfe(D_{2k}^+),\; \rmMp_n^\bfe(D_{2k}^-)\} \] 
by setting  
\begin{align*}
\rmMp_n^\bfe\Big(D_{2k}^{(-1)^{(n-1)/2}}\Big)&=\frkD^{(n)}_{(2k+n)/2}, & 
\rmMp_n^\bfe\Big(D_{2k}^{(-1)^{(n+1)/2}}\Big)&=\bar\frkD^{(n)}_{(2k+n)/2}. 
\end{align*}


\begin{conjecture}\label{coj:11}
If $\pi\simeq\otimes'_v\pi_v$ is an irreducible cuspidal automorphic representation of $\PGL_2(\AA)$ such that $\pi_v\simeq D_{2\kap_v}$ for $v\in\frkS_\infty$, then $\otimes_v'\rmMp_n^{\psi_v}(\pi_v^{\eps_v})$ has multiplicity $\frac{1}{2}\bigl(1+\vep\left(\frac{1}{2},\pi\right)\prod_v\eps_v\bigl)$ in $\scra_\cusp(\Mp(W_n))$. 
\end{conjecture}

When $n=1$, Conjecture \ref{coj:11} is the elegant fundamental result of Waldspurger \cite{W2}. 
One source of motivation for the present work is to generalize this result to high rank groups. 

\begin{theorem_a}
If none of $\pi_\frkp$ is supercuspidal 
and if $\kap_v>\frac{n}{2}$ for all $v\in\frkS_\infty$, then Conjecture \ref{coj:11} is true. 
\end{theorem_a}

The genuine part of $\scra_\cusp(\Mp(W_1))$ is exhausted by the packets of this type 
and elementary cuspidal theta series. 
But when $n>1$, those representations form a very small, exceptional class of cuspidal automorphic representations. 
In particular, they are CAP with respect to $\scrp_n$ and not tempered at any finite primes.  
If $\eps_v=+$ for all $v\in\frkS_\infty$, then elements which correspond to a lowest weight vector at every real place are Hilbert-Siegel cusp forms of weight $\bigl(\kap_v+\frac{n}{2}\bigl)_{v\in\frkS_\infty}$. 
Their Fourier coefficients will have number theoretical interest (cf. Theorem \ref{thm:11}). 

\begin{remark}\label{rem:11} 
Denote by $\frkS_\pi^-$ the set of nonarchimedean primes $\frkp$ such that $\pi_\frkp$ is the Steinberg representation. 
Let $\eps_n(\pi_v)$ be defined by 
\[\eps_n(\pi_v)=\begin{cases} 
(-1)^{n(n-1)/2}&\text{for $v\in\frkS_\infty$, } \\
-&\text{for $v\in\frkS_\pi^-$, } \\ 
+&\text{for $v\notin\frkS_\infty\cup\frkS_\pi^-$.} 
\end{cases}\]
Provided that $\vep\bigl(\frac{1}{2},\pi\bigl)\prod_v\eps_n(\pi_v)=1$, Tamotsu Ikeda and the author \cite{IY} have constructed the embedding of $\otimes_v'\rmMp_n^{\psi_v}(\pi_v^{\eps_n(\pi_v)})$ into $\scra_\cusp(\Mp(W_n))$ without the restriction on $\kap_v$ by means of its explicit Fourier expansion. 
\end{remark}

In this paper we will realize those cuspidal automorphic representations in the packet as theta liftings from orthogonal groups of quadratic forms of rank $2n+1$. 
These two different constructions will yield interesting identities of cusp forms. 
If $\prod_v\vep_v=1$, then the Minkowski-Hasse theorem gives a unique equivalence class of quadratic spaces $V_n^\vep$ of dimension $2n+1$ and discriminant $1$ which satisfy the following conditions: 
\begin{itemize}
\item $V_n^\vep(F_v)$ has signature $(2n,1)$ if $v\in\frkS_\infty$ and $\vep_v=(-1)^{(n-1)/2}$; 
\item $V_n^\vep(F_v)$ is negative definite if $v\in\frkS_\infty$ and $\vep_v=(-1)^{(n+1)/2}$; 
\item $V_n^\vep(F_\frkp)$ is split or not according as $\vep_\frkp=+$ or $\vep_\frkp=-$. 
\end{itemize}

For a technical reason we suppose that $\kap_v>\frac{n}{2}$ for all $v\in\frkS_\infty$. 
When $v\in\frkS_\infty$, we let $\rmSO_n^{\vep_v}(\pi_v)$ be the discrete series representation of $\SO(V_n^\vep,F_v)$ with Harish-Chandra parameter $\bigl(\kap_v+\frac{n}{2}-1,\kap_v+\frac{n}{2}-2,\dots,\kap_v-\frac{n}{2}\bigl)$. 
Let $\rmSO_n^{\vep_\frkp}(\pi_\frkp)$ be the unique irreducible quotient of  
\[\Ind^{\SO(V^\vep_n,F_\frkp)}_{Q^{\vep_\frkp}_n}(\pi_\frkp\otimes\alp_\frkp^{(n-1)/2})\boxtimes(\pi_\frkp\otimes\alp_\frkp^{(n-3)/2})\boxtimes\cdots\boxtimes(\pi_\frkp\otimes\alp_\frkp)\boxtimes\pi^{\vep_\frkp}_\frkp, \] 
where $Q^{\vep_\frkp}_n$ is a parabolic subgroup of $\SO(V^\vep_n,F_\frkp)$ with Levi subgroup 
\[\GL_2(F_\frkp)\times\cdots\times\GL_2(F_\frkp)\times \SO(V_1^\vep,F_\frkp). \]
We here set $\rmSO_n^-(\pi_\frkp)=0$ if $\pi_\frkp$ is an irreducible principal series. 

In Theorem \ref{thm:53} we shall show that the following conjecture is equivalent to Conjecture \ref{coj:11}, provided that $\kap_v>\frac{n}{2}$ for every $v\in\frkS_\infty$. 

\begin{conjecture}\label{coj:12}
If $\pi$ is an irreducible cuspidal automorphic representation of $\PGL_2(\AA)$ such that $\pi_v\simeq D_{2\kap_v}$ with $\kap_v>\frac{n}{2}$ for every $v\in\frkS_\infty$, then $\otimes_v'\rmSO_n^{\vep_v}(\pi_v)$ has multiplicity $1$ in the space of cusp forms on $\SO(\calv_n^\vep,\AA)$. 
\end{conjecture}

One of key ingredients is the special case of Conjecture \ref{coj:11} proved in \cite{IY} (see Remark \ref{rem:11}). 
By applying the theta lifts back and forth between the metaplectic and odd orthogonal groups, as Waldspurger did in \cite{W2}, we can deduce the following special case of Conjecture \ref{coj:12} and Theorem \ref{thm:52}. 

\begin{theorem_b}
If none of $\pi_\frkp$ is supercuspidal, then Conjecture \ref{coj:12} is true. 
\end{theorem_b}

The nonvanishing of theta liftings can be determined by the local Howe correspondence and the central value of a standard $L$-function. 
Extensions of the diagram above to the case of general $W_n$ have been obtained in the real case by Adams and Barbasch \cite{AB} and in the $\frkp$-adic case by Gan and Savin \cite{GS}. 
The $L$-function was determined by the author \cite{Y3} in complete generality. 
More serious is the obstruction of its possible vanishing at the central point. 
As for the representations we specifically consider, we can overcome this difficulty by appealing to a result of Waldspurger \cite[Theorem 4]{W2} (or Friedberg and Hoffstein \cite{FH}). 

Conjectures \ref{coj:11} and \ref{coj:12} are special cases of Arthur's conjecture for metaplectic and (non quasisplit) odd orthogonal groups (see \cite{G3}), which leads to a partition of the set of automorphic representations into packets of various types. 
Though an appropriate trace formula will establish more comprehensive results, 
our proof is elementary and gives more precise information. 

\medskip

Next we take $n$ to be even. 
Let $\SK_2(\pi_v)=\{\SK_2^+(\pi_v),\;\SK_2^-(\pi_v)\}$ be the Saito-Kurokawa packet for $\PGSp_2(F_v)$. 
It was first constructed by Piatetski-Shapiro \cite{PS2} using the dual pair $\Mp(W_1)_v\times\PGSp_2(F_v)$ and then reconstructed by Schmidt \cite{Sc} using the dual pair $\GSp_2(F_v)\times\GO(D_v)$:

\begin{picture}(300,50)
\put(222,36){$\PGL_2(F_v)\times\PGL_2(F_v)$}
\put(247,3){$\mathrm{P}D_v^\times\times\mathrm{P}D_v^\times$}
\put(15,19){$\Mp(W_1)_v$}
\put(62,23){\vector(1,0){40}}
\put(70,13){\small Howe}
\put(105,19){$\PGSp_2(F_v)$}
\put(220,39){\vector(-4,-1){60}}
\put(220,6){\vector(-4,1){60}}
\put(277,13){\vector(0,1){18}}
\put(177,36){\small Howe}
\put(177,4){\small Howe}
\end{picture}

\noindent The representation $\SK_2^\pm(\pi_v)$ is the theta lift of $\rmMp_1^{\psi_v}(\pi_v^\pm)$ and the theta lift of $\pi_v^\pm\boxtimes\1$. 
Though the similitude groups are more suitable, we here consider the isometry groups. 
We denote the maximal semisimple quotient of 
\[\Ind^{Sp_n(F_\frkp)}_{\scrq_n}(\pi_\frkp\otimes\alp_\frkp^{(n-1)/2})\boxtimes(\pi_\frkp\otimes\alp_\frkp^{(n-3)/2})\boxtimes\cdots\boxtimes(\pi_\frkp\otimes\alp_\frkp^{3/2})\boxtimes\SK_2^\pm(\pi_\frkp)|_{Sp_2(F_\frkp)} \] 
by $\rmSp^\pm_n(\pi_\frkp)$, where 
$\scrq_n$ is a parabolic subgroup of $Sp_n$ whose Levi subgroup is isomorphic to $\GL_2(F_\frkp)\times\cdots\times\GL_2(F_\frkp)\times Sp_2(F_\frkp)$. 
Set 
\[Sp_n^{(-1)^{n/2}}(D_{2k})=\frkD_{(2k+n)/2}^{(n)}\oplus\bar\frkD_{(2k+n)/2}^{(n)}. \]
In the following theorem we require the nonvanishing of the central value so the result is less complete than the metaplectic case (cf. Proposition \ref{prop:91}).

\begin{theorem_d}
Let $\pi$ be an irreducible cuspial automorphic representation of $\PGL_2(\AA)$ such that $\pi_v\simeq D_{2\kap_v}$ with $\kap_v\geq\frac{n}{2}$ for every $v\in\frkS_\infty$. 
If none of $\pi_\frkp$ is supercuspidal and $L\bigl(\frac{1}{2},\pi\bigl)\neq 0$, then the multiplicity of irreducible summands of $(\otimes_{v\in\frkS_\infty}\rmSp_n^{(-1)^{n/2}}(\pi_v))\otimes(\otimes_\frkp'\rmSp_n^{\eps_\frkp}(\pi_\frkp))$ in the space of cusp forms on $Sp_n(\AA)$ is $\frac{1}{2}(1+(-1)^{dn/2}\prod_\frkp\eps_\frkp)$. 
\end{theorem_d}

If $\prod_\frkp\vep_\frkp=(-1)^{dn/2}$, then there is a totally positive definite quadratic space $\calv_n^\vep$ of dimension $2n$ and discriminant $1$  such that $\calv_n^\vep(F_\frkp)$ is split or not according as $\vep_\frkp=+$ or $\vep_\frkp=-$. 
Let $\mathrm{S}Q_n^{\vep_\frkp}$ denote a parabolic subgroup of $\SO(\calv_n^\vep,F_\frkp)$ with Levi subgroup 
\[\GL_2(F_\frkp)\times\cdots\times\GL_2(F_\frkp)\times\SO(\calv_2^\vep,F_\frkp). \]
We write $\rmO_n^{\vep_\frkp}(\pi_\frkp)$ for the maximal semisimple quotient of  
\[\Ind^{\O(\calv_n^\vep,F_\frkp)}_{\mathrm{S}Q^{\vep_\frkp}_n}(\pi_\frkp\otimes\alp_\frkp^{(n-1)/2})\boxtimes(\pi_\frkp\otimes\alp_\frkp^{(n-3)/2})\boxtimes\cdots\boxtimes(\pi_\frkp\otimes\alp_\frkp^{3/2})\boxtimes(\pi_\frkp^{\vep_\frkp}\boxtimes\1)|_{\SO(\calv_2^\vep,F_\frkp)} \]
When $k\geq\frac{n}{2}$, we denote by $\rmO_n^{(-1)^{n/2}}(D_{2k})$ the irreducible representation of $\O(2n,\RR)$ whose highest weight is $\bigl(k-\frac{n}{2},k-\frac{n}{2},\dots,k-\frac{n}{2},0,\dots,0\bigl)$, where $k-\frac{n}{2}$ occurs $n$ times (see Section \ref{sec:3} for this notation). 

\begin{theorem_c}
Let $\pi$ be an irreducible cuspial automorphic representation of $\PGL_2(\AA)$ such that $\pi_v\simeq D_{2\kap_v}$ with $\kap_v\geq\frac{n}{2}$ for $v\in\frkS_\infty$. 
If none of $\pi_\frkp$ is supercuspidal and $\prod_\frkp\vep_\frkp=(-1)^{dn/2}$, then each irreducible summand of 
\[\O^\vep_n(\pi)\simeq(\otimes_{v\in\frkS_\infty}\rmO_n^{(-1)^{n/2}}(\pi_v))\otimes(\otimes_\frkp'\rmO_n^{\vep_\frkp}(\pi_\frkp))\] 
occurs in the space of automorphic forms on $\O(\calv_n^\vep,\AA)$ with multiplicity one. 
\end{theorem_c}


As an application, Theorem \ref{thm:101} identifies the Duke-Imamoglu-Ikeda lift with a sum of theta functions, which produces an identity analogous to the Siegel-Weil formula. 
Since $Sp_{2k}^{\eps_{2k}(\pi_\frkp)}(\pi_\frkp)$ is equivalent to a subrepresentation of a degenerate principal series whose degenerate Whittaker functional $w_\frkp^\xi$ is unique up to scalar and can be constructed canonically as a Jacquet integral, the $\xi$th Fourier coefficient of a Hilbert-Siegel cusp form $\vph=\otimes_v\vph_v\in\otimes_v'Sp_{2k}^{\eps_{2k}(\pi_v)}(\pi_v)$ is built out of the local quantity $w^\xi_\frkp(\vph_\frkp)$ and the $\det(2\xi)$th Fourier coefficient of a Hilbert cusp form in $\otimes_v'\Mp_1^{\psi_v}(\pi_v^{\eps_1(\pi_v)})$. 
On the other hand, we obtain $\vph$ as a theta lift from a definite orthogonal group, which involves theta functions and hence the position of lattice points, etc. 
Our identities thus contain a link between local and global information. 

We denote the integer ring of $F$ by $\frko$ and the different of $F/\QQ$ by $\frkd$. Put  
\begin{align*}
\Sym_j&=\{z\in\Mat_j|\trs z=z\}, &
\scrr_j&=\{\xi\in\Sym_j(F)|\tr(\xi z)\in\frko\text{ for }z\in \Sym_j(\frko)\}. 
\end{align*}
When $Q:\scrl\times\scrl\to\frko$ is a totally positive definite quadratic form, it is one of the classical tasks of number theory to determine the number of representations of $\xi\in\Sym_j(F)$ by $\scrl$   
\[N(\scrl,\xi)=\sharp\{(x_1,\dots,x_j)\in\scrl^j\;|\;Q(x_a,x_b)=2\xi_{ab}\,(1\leq a,b\leq j)\}. \]
Let $\Xi_i$ be the finite set of isometry classes of totally positive definite even unimodular lattices of rank $2i$. 
To solve this problem for $\scrl\in\Xi_i$, it is sufficient to determine the sum 
\[R(\xi,f)=\sum_{L\in\Xi_i}f(L)\frac{N(L,\xi)}{\sharp\O(L)}\]
for all Hecke eigenfunctions $f:\Xi_i\to\CC$. 
It is important to note that $R(\xi,f)$ appears in the $\xi$th Fourier coefficient of the theta lift of $f$. 
When $f:\Xi_{2k}\to\CC$ is associated to an irreducible everywhere unramified cuspidal automorphic representation of $\PGL_2(\AA)$ with parallel minimal weights $\pm 2k$ in the sense of Theorem \ref{thm:91}, we will give a product formula for $R(\xi,f)$ with $\xi\in\scrr_{2k}$. 

We denote the integer ring of $F_\frkp$ by $\frko_\frkp$ and the cardinality of the residue field $\frko/\frkp$ by $q_\frkp$. 
The norm and the order of a fractional ideal of $\frko$ are defined by $\frkN(\frkp^i)=q_\frkp^i$ and $\ord_\frkp\frkp^j=j$. 
For $\eta\in F^\times$ we write $\hat\chi^\eta:\AA^\times/F^\times\to\{\pm 1\}$ for the character corresponding to $F(\sqrt{\eta})/F$ via class field theory, denote its conductor by $\frkd^\eta$ and put 
\begin{align*}
f^\eta_\frkp&=\frac{1}{2}(\ord_\frkp\eta-\ord_\frkp\frkd^\eta)\in\ZZ, & 
\frkf_\eta&=\sqrt{\frac{|N_{F/\QQ}(\eta)|}{\frkN(\frkd^\eta)}}\in\QQ^\times. 
\end{align*}
We put $\del_\frkp(\eta)=1$ if $\sqrt{\eta}\in F_\frkp$, $\del_\frkp(\eta)=-1$ if $F_\frkp(\sqrt{\eta})$ is an unramified quadratic extension of $F_\frkp$, and $\del_\frkp(\eta)=0$ if $F_\frkp(\sqrt{\eta})$ is a ramified quadratic extension of $F_\frkp$. 
We define $\Psi_\frkp(\eta,X)\in\CC[X+X^{-1}]$ by 
\[\Psi_\frkp(\eta,X)=\begin{cases} 
\frac{X^{f_\frkp^\eta+1}-X^{-f_\frkp^\eta-1}}{X-X^{-1}}+q_\frkp^{-1/2}\del_\frkp(\eta)\frac{X^{f_\frkp^\eta}-X^{-f_\frkp^\eta}}{X-X^{-1}}&\text{if $f_\frkp^\eta\geq 0$, }\\
0 &\text{if $f_\frkp^\eta<0$. }
\end{cases}\]
For $\xi\in\Sym_{2m}(F)\cap\GL_{2m}(F)$ we set 
\begin{align*}
D(\xi)&=(-1)^m\det(2\xi), & 
\del_\frkp(\xi)&=\del_\frkp(D(\xi)), & 
f^\xi_\frkp&=f_\frkp^{D(\xi)}, &   
\frkf_\xi&=\frkf_{D(\xi)}.  
\end{align*}

The Siegel series associated to $\xi\in\scrr_{2m}\cap\GL_{2m}(F)$ and $\frkp$ is defined by 
\[b_\frkp(\xi,s)=\sum_{z\in\Sym_{2m}(F_\frkp)/\Sym_{2m}(\frko_\frkp)} \psi'_\frkp(-\tr(\xi z))\nu[z]^{-s}, \]
where $\nu[z]=[z\frko_\frkp^{2m}+\frko_\frkp^{2m}:\frko_\frkp^{2m}]$ and $\psi'_\frkp$ is an arbitrarily fixed additive character on $F_\frkp$ of order zero. 
We define the polynomial $\gam_\frkp(\xi,X)\in\ZZ[X]$ by 
\[\gam_\frkp(\xi,X)=\frac{1-X}{1-\del_\frkp(\xi)q_\frkp^mX}\prod_{j=1}^m(1-q_\frkp^{2j}X^2). \]
There exists a monic polynomial $F_\frkp(\xi,X)\in\ZZ[X]$ such that 
\begin{align*}
b_\frkp(\xi,s)&=F_\frkp(\xi,q^{-s}_\frkp)\gam_\frkp(\xi,q_\frkp^{-s}), & 
F_\frkp(\xi,X)=q_\frkp^{2m+1}X^{2f^\xi_\frkp}F_\frkp(\xi,q_\frkp^{-2m-1}X^{-1})
\end{align*} 
(see \cite{Kat,I5}). 
We define $\wtl{F}_\frkp(\xi,X)\in\CC[X+X^{-1}]$ by 
\[\wtl{F}_\frkp(\xi,X)=X^{-f^\xi_\frkp}F_\frkp(\xi,q_\frkp^{-(2m+1)/2}X). \]
The Kohnen plus subspace $S^+_{(2k+1)/2}$ of the space of Hilbert cusp forms of weight $k+\frac{1}{2}$ with respect to the congruence subgroup $\Gam_1[\frkd^{-1},4\frkd]$ (see (\ref{tag:101}) for its definition) is defined in \cite{Ko} if $F=\QQ$ and in Definition 13.2 of \cite{HI} in general. 
One obtains the following result by taking $\xi$th Fourier coefficients of the identity of Corollary \ref{cor:101}. 

\begin{theorem}\label{thm:11}
Let $\pi\simeq D_{2k}^{\boxtimes d}\otimes(\otimes'_\frkp I(\alp_\frkp^{s_\frkp},\alp_\frkp^{-s_\frkp}))$ be an irreducible cuspidal automorphic representation of $\PGL_2(\AA)$. 
Assume that $dk$ is even. 
Let $f:\Xi_{2k}\to\CC$ be a common eigenfunction of all Hecke operators whose standard $L$-function is $\prod_{j=1}^{2k}L\bigl(s+k+\frac{1}{2}-j,\pi\bigl)$. 
\begin{enumerate}
\renewcommand\labelenumi{(\theenumi)}
\item If $L\bigl(\frac{1}{2},\pi\bigl)=0$, then $R(\xi,f)=0$ for every $\xi\in\scrr_{2k}$.  
\item If $L\bigl(\frac{1}{2},\pi\bigl)\neq 0$, then $R(\xi,f)=0$ unless $\xi\in\scrr_{2k}\cap\Sym_{2k}^+$, and 
\[R(\xi,f)=c(\det(2\xi))\frkf_\xi^{(2k-1)/2}\prod_\frkp\wtl{F}_\frkp(\xi,q_\frkp^{s_\frkp}) \]
for every $\xi\in\scrr_{2k}\cap\Sym_{2k}^+$, where the constant $c(\eta)$ satisfies $c(\eta a^2)=c(\eta)\prod_{v\in\frkS_\infty}\sgn_v(a)^k$ for every $a\in F^\times$, and 
\[0\neq h(\calz)=\sum_\eta c(\eta)\bfe_\infty(\eta \calz)\frkf_{(-1)^k\eta}^{(2k-1)/2}\prod_\frkp\Psi_\frkp((-1)^k\eta,q_\frkp^{s_\frkp})\in S^+_{(2k+1)/2}. \]
\end{enumerate}
\end{theorem}

\begin{remark}\label{rem:12}
There are three further observations to make here. 
\begin{enumerate}
\renewcommand\labelenumi{(\theenumi)}
\item\label{rem:121} Since $f$ and $h$ are determined uniquely up to scalar, one gets relations between the ratios of the various $R(\xi,f)$ and those of $c(\det(2\xi))$. 
For example, if $\det\xi=\det\xi'$ and $\frkf_\xi=\frkf_{\xi'}=1$, then $R(\xi,f)=R(\xi',f)$. 
\item\label{rem:122} When $F=\QQ$, Kohnen and Zagier \cite{KZ} have given an explicit formula for $c(\eta)^2$. 
It is noteworthy that our formula can give information on the sign of $c(\eta)$. 
See Theorem 12.3 of \cite{HI} for an extension of the Kohnen-Zagier formula to Hilbert cusp forms. 
In particular, $c(\eta)\neq 0$ if and only if $L\bigl(\frac{1}{2},\pi\otimes\hat\chi^{(-1)^k\eta}\bigl)\neq 0$. 
\item\label{rem:123}  If $f\equiv 1$, then the Siegel formula \cite{Si} says that for $\xi\in\scrr_{2k}\cap\Sym_{2k}^+$
\[\frac{R(\xi,1)}{\sum_{L\in\Xi_{2k}}\frac{1}{E(L)}}=
\frac{N_{F/\QQ}(\det\xi)^{(2k-1)/2}}{\frkN(\frkd)^{k(6k-1)/2}}\prod_{j=0}^{2k-1}\frac{\pi^{d(4k-j)/2}}{\Gam\bigl(\frac{4k-j}{2}\bigl)^d}\prod_\frkp b_\frkp(\xi,2k). \]
\end{enumerate} 
\end{remark}

\subsection*{Acknowledgement}
The author is partially supported by JSPS Grant-in-Aid for Young Scientists (B) 26800017. 
We would like to thank Tamotsu Ikeda for stimulating discussions. 

\section*{Notation}

We denote by $\NN$, $\ZZ$, $\QQ$, $\RR$, $\CC$, $\RR^\times_+$, $\SS$ and $\mu_k$ the set of strictly positive rational integers, the ring of rational integers, the fields of rational, real, complex numbers, the groups of strictly positive real numbers, complex numbers of absolute value $1$ and $k^\mathrm{th}$ roots of unity. 
We define the additive character $\bfe:\RR\to\SS$ by $\bfe(x)=e^{2\pi\iu x}$ and the sign character $\sgn:\RR^\times\to \mu_2$ by $\sgn(t)=t/|t|$. 
We also denote the determinant characters of orthogonal groups by $\sgn$.
When $G$ is a totally disconnected locally compact group or a real Lie group, we denote the trivial representation of $G$ by $\1_G$ and write $\cals(G)$ for the space of Schwartz-Bruhat functions on $G$. 

The ground field $F$ is a totally real number field or its completion. 
When $F$ is global, we denote the ad\`{e}le ring of $F$ by $\AA$, the set of real primes of $F$ by $\frkS_\infty$ and the $v$-completion of $F$ by $F_v$ for  each prime $v$ of $F$. 
We reserve the letter $\frkp$ for nonarchimedean primes of $F$ and do not use $\frkp$ to denote infinite primes.
When $F$ is local, we write $\Ome(F^\times)$ for the group of all continuous homomorphisms from $F^\times$ to $\CC^\times$ and let $\alp:F^\times\to\RR^\times_+$ denote the normalized absolute value. 
We define $\Re\mu$ as the unique real number such that $\mu\alp^{-\Re\mu}$ is unitary. 

When $\pi$ is a discrete series representation of $\GL_2(F)$, we write $\pi^-$ for the Jacquet-Langlands lift of $\pi$ to the multiplicative group of the quaternion division algebra over $F$. 
To be uniform, we set $\pi^+=\pi$. 

When $F$ is nonarchimedean, we associate to a pair of characters $\mu_1,\mu_2\in\Ome(F^\times)$ the representation $I(\mu_1,\mu_2)$ of $\GL_2(F)$ on the space of all smooth functions $f$ on $\GL_2(F)$ satisfying 
\[f\left(\begin{pmatrix} a_1 & b \\ 0 & a_2\end{pmatrix}g\right)=\mu_1(a_1)\mu_2(a_2)\alp(a^{}_1a_2^{-1})^{1/2}f(g)\]
for all $a_1,a_2\in F^\times$; $b\in F$ and $g\in\GL_2(F)$. 
This representation is irreducible unless $\mu^{}_1\mu_2^{-1}\in\{\alp,\alp^{-1}\}$. 
Let $\St\subset I(\alp^{1/2},\alp^{-1/2})$ stand for the Steinberg representation of $\GL_2(F)$. 

For each positive rational integer $\kap$ we denote the discrete series representation of $\PGL_2(\RR)$ with extremal weights $\pm 2\kap$ by $D_{2\kap}$. 
For $\eta\in F^\times$ we write $\hat\chi^\eta$ for the character which corresponds to $F(\sqrt{\eta})/F$ via class field theory. 
The root numbers of these representations are well-known. 

\begin{lemma}\label{lem:11}
\begin{enumerate}
\renewcommand\labelenumi{(\theenumi)}
\item\label{lem:111} If $-\frac{1}{2}<\Re\mu<\frac{1}{2}$, then $\vep\left(\frac{1}{2},I(\mu,\mu^{-1})\right)=\mu(-1)$. 
\item\label{lem:112} If we put $d_\eta=[F(\sqrt{\eta}):F]$, then $\vep\bigl(\frac{1}{2},\St\otimes\hat\chi^\eta\bigl)=(-1)^{d_\eta}\hat\chi^\eta(-1)$. 
\item\label{lem:113} $\vep\left(\frac{1}{2},D_{2\kap}\right)=(-1)^\kap$. 
\end{enumerate}
\end{lemma}

When $\pi\simeq\otimes_v'\pi^{}_v$ is an irreducible cuspidal automorphic representation of $\PGL_2(\AA)$, the set $\frkS_\pi$ consists of finite primes $\frkp$ of $F$ such that $\pi_\frkp$ is a discrete series representation of $\PGL_2(F_\frkp)$ and its subset $\frkS_\pi^-$ consists of primes $\frkp$ such that $\pi_\frkp\simeq \St_\frkp$. 
Let $G$ be a reductive algebraic group over $F$ and $\bfG$ a finite central covering of $G(\AA)$ such that $G(F)$ lifts to a subgroup of $\bfG$. 
We write $\scra_\cusp(\bfG)$ for the space of cusp forms on $G(F)\bsl\bfG$ and denote the multiplicity of an abstract representation $\vPi$ of $\bfG$ in $\scra_\cusp(\bfG)$ by $\rm\rmm_\cusp(\vPi)$.

\part{CAP representations for metaplectic groups}\label{part:1}

\section{Theta correspondence}\label{sec:2}

We review the necessary background and framework for the theta correspondence for the dual pair $\Mp(W_n)\times\O(V)$. 
For a detailed treatment one can consult \cite{G,GG,MVW,GI,Y3,GT2}. 
Let $W_n$ be a $2n$-dimensional left vector space over $F$ equipped with a nondegenerate skew symmetric form $\La\;,\;\Ra$ and $V$ an $m$-dimensional right vector space with a nondegenerate symmetric form $(\;,\;)$. 
The associated symplectic and orthogonal groups are denoted by $Sp(W_n)$ and $\O(V)$. 
We take matrix representation 
\begin{align*}
Sp_n&=Sp(W_n)=\{g\in\GL_{2n}\;|\;gJ_n\trs g=J_n\}, & 
J_n&=\begin{pmatrix} 0 & -\ono_n \\ \ono_n & 0\end{pmatrix}
\end{align*}
by choosing a Witt basis of $W_n$. 
We consider the symplectic vector space $(V\otimes W_n,(\;,\;)\otimes\La\;,\;\Ra)$ of dimension $2mn$. 
We have natural homomorphisms 
\[Sp(W_n)\hookrightarrow Sp(V\otimes W_n)\hookleftarrow \O(V)\hookleftarrow\SO(V). \] 
The groups $\mathrm{O}(V)$ and $Sp(W_n)$ form a dual pair inside $Sp(V\otimes W_n)$. 
 
Fix a place $v$ of $F$ and temporarily suppress the subscript $_v$ from notation. 
For the time being, let $\psi$ be an arbitrarily fixed nontrivial additive character on $F=F_v$.  
The metaplectic group $\Mp(W_n)$ is the nontrivial central extension of $Sp(W_n)$ by $\mu_2$.  
For a subgroup $H$ of $Sp_n$ we denote the inverse image of $H$ in $\Mp(W_n)$ by $\til H$. 

The character $\psi$ determines a Weil representation of the metaplectic double cover of the ambient symplectic group $Sp(V\otimes W_n)$. 
We obtain the representation $\ome^\psi_V$ as its pull-back to $\O(V)\times \Mp(W_n)$. 
If $m$ is even, then $\ome^\psi_V$ descends to a representation of $\O(V)\times Sp(W_n)$. 
Note that $\ome^{\psi^\eta}_V\simeq \ome^\psi_{\eta V}$ for $\eta\in F^\times$, where $\eta V$ is the space $V$ equipped with the quadratic form $\eta\cdot (\;,\;)$. 

When $m$ is odd (resp. even) and $\vPi$ is an irreducible admissible genuine (resp. nongenuine) representation of $\Mp(W_n)$, the maximal quotient of $\ome^\psi_V$ on which $\Mp(W_n)$ acts as a multiple of $\vPi$ is of the form $\vPi\boxtimes\Tht^\psi_V(\vPi)$, where $\Tht^\psi_V(\vPi)$ is a representation of $\O(V)$. 
We set $\Tht^\psi_V(\vPi)=0$ if $\vPi$ does not occur as a quotient of $\ome^\psi_V$. 
Let $\tht^\psi_V(\vPi)$ be the maximal semisimple quotient of $\Tht^\psi_V(\vPi)$. 
When $\sig$ is an irreducible admissible representation of $\O(V)$, we can analogously define representations $\Tht^\psi_n(\sig)$ and $\tht^\psi_n(\sig)$ of $\Mp(W_n)$. 
Then $\tht^\psi_V(\vPi)$ and $\tht^\psi_n(\sig)$ are either zero or irreducible, and $\vPi\mapsto\tht^\psi_V(\vPi)$ and $\sig\mapsto\tht^\psi_n(\sig)$ are bijections converse to each other on their domains by the Howe conjecture proved by Howe \cite{H} and Gan--Takeda \cite{GT2}. 

\begin{proposition}[\cite{AB,Z,GS}]\label{prop:21}
If $\vPi$ is an irreducible admissible genuine representation of $\Mp(W_n)$, then there is exactly one equivalence class of quadratic spaces $V$ of dimension $2n+1$ and discriminant $1$ such that $\tht^\psi_V(\vPi)$ is nonzero. 
\end{proposition}

When $m$ is odd, we often view $\tht^\psi_V(\vPi)$ as a representation of $\SO(V)$. 
When $m=2n+1$ and $\sig$ is an irreducible admissible representation of $\SO(V)$, exactly one extension of $\sig$ to $\O(V)$ participates in the Howe correspondence with $\Mp(W_n)$ (see \cite{AB,GS}). 
We denote the unique extension by $\hat\sig$. 

We go back to the global setting. 
Let $\ome^\psi_V\simeq\otimes'_v\ome^{\psi_v}_V$ be the global Weil representation associated to $\psi$. 
There is a natural $Sp_{mn}(F)$-invariant map $\Tht:\ome^\psi_V\to\CC$. 
Let $\vPi$ be an irreducible cuspidal automorphic representation of $\Mp(W_n)_\AA$. 
For $f\in\vPi$ and $\phi\in\ome^\psi_V$ we set 
\[\tht^\psi_V(h;f,\phi)=\int_{Sp_n(F)\bsl Sp_n(\AA)}\overline{f(g)}\Tht(\ome^\psi_V(g,h)\phi)\, \d g.\]
Then $\tht^\psi_V(f,\phi)$ is an automorphic form on $\O(V)$. 
We write $\tht^\psi_V(\vPi)$ for the subspace of the space of automorphic forms on $\O(V)$ spanned by $\tht^\psi_V(f,\phi)$ for all $f\in\vPi$ and $\phi\in\ome^\psi_V$. 
It is a simple consequence of the Howe conjecture that if $\tht^\psi_V(\vPi)$ is nonzero and contained in the space of square-integrable automorphic forms on $\O(V)$, then $\tht^\psi_V(\vPi)\simeq\otimes_v'\tht^{\psi_v}_V(\vPi_v)$.  
Given an irreducible cuspidal automorphic representation $\sig\simeq\otimes_v'\sig_v$ of $\O(V,\AA)$, one can define its theta lift $\tht^\psi_n(\sig)$ to $\Mp(W_n)_\AA$ in a similar fashion. 

Refining the work of Piatetski-Shapiro and Rallis \cite{PSR}, Lapid and Rallis \cite{LR} have developed a local theory of the doubling zeta integral and defined the local factors $\vep(s,\sig_v,\psi_v)$ and $L(s,\sig_v)$ properly. 
They satisfy a number of properties which characterize them uniquely. 
Wee Teck Gan \cite{G2} treated the analogous theory for metaplectic groups and defined the local factors $\vep(s,\vPi_v,\psi_v)$ and $L_{\psi_v}(s,\vPi_v)$. 
Define the complete $L$-functions by the Euler products $L(s,\sig)=\prod_vL(s,\sig_v)$ and $L_\psi(s,\vPi)=\prod_vL_{\psi_v}(s,\vPi_v)$. 
If $\dim V=2n+1$ and $\sig_v=\tht^{\psi_v}_V(\vPi_v)$, then Corollary 11.2 of \cite{GS} gives the identities 
\begin{align}
\vep(s,\vPi_v,\psi_v)&=\vep(s,\sig_v,\psi_v), & 
L_{\psi_v}(s,\vPi_v)&=L(s,\sig_v). \label{tag:21}
\end{align}

The following results were proved through the works of many authors. 

\begin{proposition}\label{prop:22}
\begin{enumerate}
\renewcommand\labelenumi{(\theenumi)}
\item\label{prop:221} Let $\vPi$ be an irreducible genuine cuspidal automorphic representation of $\Mp(W_n)_\AA$ such that $L_\psi(s,\vPi)$ is entire and $L_\psi\bigl(\frac{1}{2},\vPi\bigl)\neq 0$. 
Then there is a unique equivalence class of quadratic spaces $V$ of dimension $2n+1$ and discriminant $1$ such that $\tht^\psi_V(\vPi)$ is nonzero and cuspidal. 
Moreover, 
\[\rmm_\cusp(\tht^\psi_V(\vPi))=\rmm_\cusp(\vPi). \] 
\item\label{prop:222} Let $V$ be a quadratic space of dimension $2n+1$ and discriminant $1$. 
If $\sig$ is an irreducible cuspidal automorphic representation of $\SO(V,\AA)$ such that $L(s,\sig)$ is entire and $L\bigl(\frac{1}{2},\sig\bigl)\neq 0$, then there is a unique extension of $\sig$ to a cuspidal automorphic representation $\hat\sig$ of $\O(V,\AA)$ such that 
$\tht^\psi_n(\hat\sig)$ is nonzero and cuspidal. 
Moreover,  
\[\rmm_\cusp(\tht^\psi_n(\hat\sig))=\rmm_\cusp(\sig). \] 
\end{enumerate}
\end{proposition}

\begin{proof}
Theorem 10.1 of \cite{Y3} gives the quadratic space $V$ of dimension $2n+1$ and discriminant $1$ such that $\tht^\psi_V(\vPi)$ is nonzero. 
Since $L_\psi(s,\vPi)$ is entire, if $\dim U$ is odd, $\dim U\leq 2n-1$ and $\chi^U=1$, then $\tht^\psi_U(\vPi)$ is zero by Lemma 10.2 of \cite{Y3}. 
Thus $\tht^\psi_V(\vPi)$ is cuspidal by the tower property of Rallis. 
Moreover, the nonvanishing of $\tht^\psi_V(\vPi)$ depends only on $\vPi$ as an abstract representation, not on the embedding of $\vPi$ into $\scra_\cusp(\Mp(W_n))$. 
We can apply \cite[Proposition 2.14]{G} to see that the multiplicity of $\tht^\psi_V(\vPi)$ in $\tht^\psi_V(\scra_\cusp(\Mp(W_n)))$ equals $\rmm_\cusp(\vPi)$. 
Thus 
\[\rmm_\cusp(\tht^\psi_V(\vPi))\geq \rmm_\cusp(\vPi). \]

Put $\sig=\tht^\psi_V(\vPi)$. 
Then $L(s,\sig)=L_\psi(s,\vPi)$ by (\ref{tag:21}). 
By assumption $L(s,\sig)$ is entire and $L\bigl(\frac{1}{2},\sig\bigl)\neq 0$. 
We can find the extension $\hat\sig$ of $\sig$ such that $\tht^\psi_n(\hat\sig)$ is an irreducible cuspidal automorphic representation of $\Mp(W_n)_\AA$ via the same type of reasoning. 
Since $\vPi\simeq\tht^\psi_n(\hat\sig)$ by symmetry, Proposition 2.14 of \cite{G} again shows that 
\[\rmm_\cusp(\vPi)=\rmm_\cusp(\tht^\psi_n(\hat\sig))\geq \rmm_\cusp(\hat\sig)=\rmm_\cusp(\sig). \]
We have thus proved the equality. 
\end{proof}

\section{Langlands classification and theta correspondence}\label{sec:3}

Let $F$ be a local field of characteristic zero in this section. 
Define the Weil index $\gam^\psi$ as in Section \ref{sec:1}. 
If $P=MN$ is a standard parabolic subgroup of $Sp_n$ with Levi subgroup $M\simeq\GL_{n_1}\times\cdots\times\GL_{n_r}\times Sp_{n_0}$, where $n_1+\cdots+n_r+n_0=n$, $\pi_i$ is a representation of $\GL_{n_i}(F)$ and $\pi_0$ is a genuine representation of $\Mp(W_{n_0})$, then we can view 
\[\pi^\psi:=(\pi_1\otimes\gam^\psi)\boxtimes\cdots\boxtimes(\pi_r\otimes\gam^\psi)\boxtimes\pi_0\]
as a genuine representation of $\til M$ and we may consider the normalized induced representation $\Ind_{\til P}^{\Mp(W_n)}\pi^\psi$. 
When $n_1=n_2=\cdots=n_r=2$ and $n_0\leq 1$, we will write $P=\scrp_n$. 

We identify the center $Z_n$ of $Sp_n(F)$ with $\mu_2$. 
Its inverse image $\til Z_n$ is the center of $\Mp(W_n)$. 
If $\vPi$ is an irreducible genuine representation of $\Mp(W_n)$, then the central character of $\vPi$ has the form $\gam^\psi\cdot z^\psi(\vPi)$, where $z^\psi(\vPi)$ is a character of $Z_n$. 
We shall regard $z^\psi(\vPi)$ as $\pm 1$, depending on whether this character is trivial or not. 

Let $V_1^+$ and $V_1^-$ be split and anisotropic quadratic spaces of dimension $3$ and discriminant $1$, respectively.  
Let $D_-$ denote the unique quaternion division algebra over $F$. 
Put $D_+=\Mat_2(F)$. 
Recall that $\SO(V_1^\pm)\simeq\mathrm{P}D_\pm^\times$. 
Given an irreducible infinite dimensional unitary representation $\pi$ of $\PGL_2(F)$, we write $\rmMp_1^\psi(\pi^\pm)$ for the unique equivalence class of an irreducible admissible genuine representation of $\Mp(W_1)$ such that $\rmMp_1^\psi(\pi^\pm)\boxtimes\hat\pi^\pm$ is a quotient of the Weil representation $\ome^\psi_{V^\pm_1}$.  
We formally set $\rmMp_1^\psi(\pi^-)=0$ if $\pi$ is not a discrete series representation. 
For $\eta\in F^\times$ we define $\left(\frac{\eta}{\pi}\right)\in\mu_2$ by  
\[\vep\left(\frac{1}{2},\pi\otimes\hat\chi^\eta\right)=\left(\frac{\eta}{\pi}\right)\hat\chi^\eta(-1)\vep\left(\frac{1}{2},\pi\right). \] 
Theorem 2.5 of \cite{PS} says that 
\beq
\rmMp_1^{\psi^\eta}(\pi^\pm\otimes\hat\chi^\eta)=\rmMp_1^\psi\bigl(\pi^{\pm\left(\frac{\eta}{\pi}\right)}\bigl). \label{tag:31}
\eeq


First we discuss the nonarchimedean case. 
Define quadratic spaces $V_n^\pm$ of dimension $2n+1$ over $F$ by $V_n^\pm=V^\pm_1\oplus\scrh_{n-1}$, where $\scrh_j$ is the split quadratic space of dimension $2j$. 
Let $k$ be a nonnegative integer. 
Set $n=2k+1$. 
Let $P^\pm_n$ denote a parabolic subgroup of $\SO(V^\pm_n)$ with Levi subgroup \[\GL_2(F)\times\cdots\times\GL_2(F)\times \SO(V^\pm_1). \]

\begin{definition}\label{def:31}
We denote by $\rmMp^\psi_n(\pi^\pm)$ the Langlands quotient of  
\[\Ind^{\Mp(W_n)}_{\til\scrp_n}(\pi\otimes\gam^\psi\alp^k)\boxtimes(\pi\otimes\gam^\psi\alp^{k-1})\boxtimes\cdots\boxtimes(\pi\otimes\gam^\psi\alp)\boxtimes \rmMp^\psi_1(\pi^\pm) \]
 and by $\rmSO_n^\pm(\pi)$ the Langlands quotient of the standard module
\[\Ind^{\SO(V^\pm_n)}_{P^\pm_n}(\pi\otimes\alp^k)\boxtimes(\pi\otimes\alp^{k-1})\boxtimes\cdots\boxtimes(\pi\otimes\alp)\boxtimes\pi^\pm. \] 
Here $\rmMp_n^\psi(\pi^-)$ and $\rmSO_n^-(\pi)$ are interpreted as $0$ if $\pi$ is not a discrete series. 
\end{definition}

Let $x_n^\pm:\SO(V_n^\pm)\to F^\times/F^{\times 2}$ stand for the spinor norm. 

\begin{lemma}\label{lem:31}
If $\pi$ is an irreducible admissible unitary infinite dimensional representation of $\PGL_2(F)$, then 
\begin{align*}
\tht_{V_n^\pm}^\psi(\rmMp^\psi_n(\pi^\pm))&\simeq\rmSO_n^\pm(\pi), & 
z^\psi(\rmMp^\psi_n(\pi^\pm))&=\pm\vep(1/2,\pi). 
\end{align*}
More generally, for $\eta\in F^\times$, 
\beq
\tht_{V_n^\pm}^{\psi^\eta}\bigl(\rmMp^\psi_n\bigl(\pi^{\pm\left(\frac{\eta}{\pi}\right)}\bigl)\bigl)\simeq\rmSO_n^\pm(\pi\otimes\hat\chi^\eta)\simeq\rmSO_n^\pm(\pi)\otimes\hat\chi^\eta\circ x_n^\pm. \label{tag:32}
\eeq
\end{lemma}  

\begin{proof}
The Howe correspondence is compatible with the Langlands classification by Proposition \ref{prop:21} and \cite[Theorem 8.1]{GS}, which proves the first equivalence. 
Theorems 5.2 and 7.1 of \cite{GS} imply that 
\[z^\psi(\rmMp^\psi_n(\pi^\pm))=z^\psi(\rmMp^\psi_1(\pi^\pm))=\pm\vep(1/2,\pi). \]
Since $\gam^{\psi^\eta}/\gam^\psi=\hat\chi^\eta$, we can deduce the second equivalence from the expression (\ref{tag:31}). 
Twisting by the character $\hat\chi^\eta\circ x_n^\pm$ commutes with parabolic induction in view of Lemma 4.9 of \cite{Sh4}, which proves the last equivalence.  
\end{proof}
  
Let $\psi=\bfe$. 
The symbol $\frkD_\ell^{(n)}$ (resp. $\bar\frkD_\ell^{(n)}$) denotes the lowest (resp. highest) weight representation of $\Mp(W_n)$ with lowest (resp. highest) $K$-type $(\det)^\ell$ (resp. $(\det)^{-\ell}$) for $\ell\in\frac{1}{2}\NN$. 
For an odd natural number $n$ and $\kap\in\NN$ we define genuine representations $\rmMp^\psi_n(D_{2\kap}^\pm)$ of $\Mp(W_n)$ by 
\begin{align*}
\rmMp_n^\psi\Big(D_{2\kap}^{(-1)^{(n-1)/2}}\Big)&=\frkD^{(n)}_{(2\kap+n)/2}, & 
\rmMp_n^\psi\Big(D_{2\kap}^{(-1)^{(n+1)/2}}\Big)&=\bar\frkD^{(n)}_{(2\kap+n)/2}. 
\end{align*}

Irreducible representations of $\O(n)$ are parameterized by elements $\lam\in\ZZ^n$ of the form $\lam=(a_1,\dots,a_k,0,\dots,0)$ or $\lam=(a_1,\dots,a_k,1,\dots,1,0,\dots,0)$, where $k\leq\bigl[\frac{n}{2}\bigl]$ and $a_1\geq a_2\geq\dots\geq a_k\geq 1$. 
In the latter case $1$ occurs $n-2k$ times. 
The symbol $\tau(\lam)$ denote the finite-dimensional representation of $\O(n)$ corresponding to $\lam$. 
Note that 
\beq
\tau(a_1,\dots,a_k,1,\dots,1,0,\dots,0)\simeq\tau(a_1,\dots,a_k,0,\dots,0)\otimes\sgn. \label{tag:33}
\eeq
We define the equivalence classes of real quadratic spaces $V_n^\pm$ of dimension $2n+1$ so that $V^{(-1)^{(n-1)/2}}_n$ has signature $(2n,1)$ and $V^{(-1)^{(n+1)/2}}_n$ is negative definite. 
For $\kap>\frac{n}{2}$ we define $\rmSO_n^\pm(D_{2\kap})$ as the discrete series representation of $\SO(V_n^\pm)$ with Harish-Chandra parameter 
\[\left(\kap+\frac{n}{2}-1,\kap+\frac{n}{2}-2,\dots,\kap-\frac{n}{2}\right). \] 

\begin{lemma}[\cite{AB}]\label{lem:32}
Suppose that $n$ is odd, $\psi=\bfe$ and $\kap>\frac{n}{2}$. 
\begin{enumerate}
\renewcommand\labelenumi{(\theenumi)}
\item\label{lem:321} If $\eta>0$, then $\tht_{V_n^\pm}^{\psi^\eta}(\rmMp^\psi_n(D_{2\kap}^\pm))\simeq\rmSO_n^\pm(D_{2\kap})$. 
\item\label{lem:322} If $\eta<0$, then $\tht_{V_n^\mp}^{\psi^\eta}(\rmMp^\psi_n(D_{2\kap}^\pm))\simeq\rmSO_n^\mp(D_{2\kap})$. 
\item\label{lem:323} $\rmSO_n^\vep(D_{2\kap})\otimes\sgn\circ x^\vep_n\simeq\rmSO_n^\vep(D_{2\kap})$. 
\end{enumerate}
\end{lemma}

\begin{remark}\label{rem:31}
Since $D_{2\kap}\otimes\sgn\simeq D_{2\kap}$ and $\bigl(\frac{\eta}{D_{2\kap}}\bigl)=\sgn(\eta)$, the equivalence (\ref{tag:32}) remains valid when $F=\RR$ and $\pi\simeq D_{2\kap}$. 
\end{remark}

\begin{proof}
If $\ell$ is a half integer greater than $n$, then $\frkD^{(n)}_\ell$ and $\bar\frkD^{(n)}_\ell$ are discrete series representations and their Harish-Chandra parameters are 
\begin{align*}
&(\ell-1,\ell-2,\dots,\ell-n) & &\text{and} & 
&(n-\ell,n-1-\ell,\dots,1-\ell), 
\end{align*}
respectively. 
Theorem 3.3 of \cite{AB} shows that $\tht^{\psi^t}_{V_n^\vep}(\frkD^{(n)}_\ell)$ and $\tht^{\psi^t}_{V_n^{-\vep}}(\bar\frkD^{(n)}_\ell)$ are discrete series representations whose Harish-Chandra parameters are $(\ell-1,\ell-2,\dots,\ell-n)$, where $\vep=\sgn(t)(-1)^{(n-1)/2}$. 
Letting $\ell=\kap+\frac{n}{2}$, we obtain (\ref{lem:321}), (\ref{lem:322}). 
If $\SO(V^\vep_n)$ is topologically connected, then $x^\vep_n$ is trivial. 
We may therefore assume that $\vep=(-1)^{(n-1)/2}$ in the proof of (\ref{lem:323}). 
The minimal $K$-type of $\rmSO_n^\vep(D_{2\kap})$ is
\[\left(\kap-\frac{n-1}{2},\kap-\frac{n-1}{2},\dots,\kap-\frac{n-1}{2},0,\dots,0\right)\in\ZZ^{2n},\]
where $0$ occurs $n$ times. 
Thus $\rmSO_n^\vep(D_{2\kap})$ and $\rmSO_n^\vep(D_{2\kap})\otimes\sgn\circ x_n^\vep$ have the same minimal $K$-type by (\ref{tag:33}). 
Since they clearly have the same infinitesimal character, they are equivalent (see the note added in Section 3 of \cite{AB}). 
\end{proof}

\section{Degenerate principal series and Langlands classification}\label{sec:4}

We let $F$ be a nonarchimedean local field in this section. 
For $k\leq n$ we denote by $P_k^n$ the parabolic subgroup of $Sp_n$ stabilizing a totally isotropic subspace of $W_n$ of dimension $k$. 
Composing a character $\mu\in\Ome(F^\times)$ by $\det$ gives a character of $M_n^n(F)\simeq\GL_n(F)$. 
One can define the degenerate principal series representation $I_n^\psi(\mu)=\Ind^{\Mp(W_n)}_{\til P_n^n}\mu\gam^\psi$. 

\begin{lemma}[\cite{Sw}]\label{lem:41}
Assume that $n$ is odd. 
Let $\mu\in\Ome(F^\times)$ and $\eta\in F^\times$. 
\begin{enumerate}
\renewcommand\labelenumi{(\theenumi)}
\item\label{lem:411} If $-\frac{1}{2}<\Re\mu<\frac{1}{2}$, then $I_n^\psi(\mu)$ is irreducible. 
\item\label{lem:412} $I_n^\psi(\hat\chi^\eta\alp^{1/2})$ has a unique irreducible subrepresentation, which is denoted by $\St^{\psi^\eta}_n$. 
\end{enumerate}
\end{lemma}

When $n=m=1$ and $V$ has discriminant $1$, the representation $\ome^{\psi^\eta}_V$ is realized in $\cals(F)$ and the direct sum $\ome^{\psi^\eta}_+\oplus\ome^{\psi^\eta}_-$ of two irreducible representations: $\ome^{\psi^\eta}_+$ (resp. $\ome^{\psi^\eta}_-$) consists of even (resp. odd) functions in $\cals(F)$. 

\begin{proposition}[{\cite[Propositions 4, 5, 8, 9]{W2}}]\label{prop:41} 
\begin{enumerate}
\renewcommand\labelenumi{(\theenumi)}
\item\label{prop:411} If $\mu\in\Ome(F)$ and $-\frac{1}{2}<\Re\mu<\frac{1}{2}$, then $\rmMp_1^\psi(I(\mu,\mu^{-1})^+)\simeq I^\psi_1(\mu)$. 
\item\label{prop:412} If $\pi\simeq\St$, then $\rmMp_1^\psi(\pi^+)\simeq\ome^\psi_-$ and $\rmMp_1^\psi(\pi^-)\simeq\St^\psi_1$. 
\item\label{prop:413} If $\pi\simeq\St\otimes\hat\chi^\eta$ with $\eta\notin F^{\times 2}$, then $\rmMp_1^\psi(\pi^+)\simeq\St^{\psi^\eta}_1$ and $\rmMp_1^\psi(\pi^-)\simeq \ome^{\psi^\eta}_-$. 
\item\label{prop:414} If $\psi=\bfe$, then $\rmMp_1^\psi(D_{2k}^+)\simeq\frkD^{(1)}_{(2k+1)/2}$ and $\rmMp^\psi_1(D_{2k}^-)\simeq\bar\frkD^{(1)}_{(2k+1)/2}$. 
\end{enumerate}
\end{proposition}

\begin{lemma}\label{lem:44}
Let $n$ be odd and $\mu\in\Ome(F^\times)$. 
\begin{enumerate}
\renewcommand\labelenumi{(\theenumi)}
\item\label{lem:441} If $-\frac{1}{2}<\Re\mu<\frac{1}{2}$, then $I^\psi_n(\mu)\simeq \rmMp_n^\psi(I(\mu,\mu^{-1})^+)$. 
\item\label{lem:442} If $\eta\in F^\times\setminus F^{\times2}$ and $\pi\simeq\St\otimes\hat\chi^\eta$, then $\St^{\psi^\eta}_n\simeq \rmMp_n^\psi(\pi^+)$. 
\item\label{lem:443} $\St^\psi_n\simeq \rmMp_n^\psi(\St^-)$. 
\end{enumerate}
\end{lemma}

\begin{proof}
These results are proved in \cite{Y4}. 
\end{proof}

\section{CAP representations for odd orthogonal groups}\label{sec:5}

We now work over a totally real field $F$. 
Put $d=[F:\QQ]$. 
Let $\pi\simeq\otimes_v'\pi^{}_v$ be an irreducible cuspidal automorphic representation of $\PGL_2(\AA)$ such that $\pi_v\simeq D_{2\kap_v}$ for $v\in\frkS_\infty$. 
Recall that $\ell_\pi=\shp\frkS_\pi$ and $\ell^-_\pi=\shp\frkS^-_\pi$, where $\frkS_\pi$ and $\frkS_\pi^-$ are defined in the notation section. 
Throughout this section $n$ is odd. 
We associate to each function $\eps:\frkS_\infty\cup\frkS_\pi\to\mu_2$ the irreducible admissible genuine representation $\rmMp^\psi_n(\pi^\eps)=\otimes'_v \rmMp^{\psi_v}_n(\pi_v^{\eps_v})$ of $\Mp(W_n)_\AA$ and let $\rmMp_n^\psi(\pi)$ be the set of $2^{d+\ell_\pi}$ such representations of $\Mp(W_n)_\AA$. 

Recall the identity $z^{\psi_\frkp}(\rmMp^{\psi_\frkp}_n(\pi_\frkp^{\eps_\frkp}))=\eps_\frkp\vep\left(\frac{1}{2},\pi_\frkp\right)$ stated in 
Lemma \ref{lem:31}. 
This remains true in the archimedean case by \cite[(5.2)]{IY} and Lemma \ref{lem:11}(\ref{lem:113}). 
Therefore roughly half of the elements in $\rmMp_n^\psi(\pi)$ cannot be automorphic for the trivial reason. 

\begin{lemma}\label{lem:51}
Notation being as above, if $\prod_v\eps_v\neq\vep\left(\frac{1}{2},\pi\right)$, then $\rmMp^\psi_n(\pi^\eps)$ cannot appear in the space of automorphic forms on $\Mp(W_n)_\AA$. 
\end{lemma}

We employ the symbol $\eps_n(\pi_v)$ defined in Remark \ref{rem:11}. 
Put 
\[\Ik_n^\psi(\pi)=\otimes_v'\rmMp_n^{\psi_v}\bigl(\pi_v^{\eps_n(\pi_v)}\bigl)\in \rmMp_n^\psi(\pi). \]
Ikeda and the author \cite{IY} have proved the special case of Conjecture \ref{coj:11} where $\rmMp^\psi_n(\pi^\eps)$ takes the form $\Ik_n^\psi(\pi)$. 

\begin{proposition}\label{prop:51}
Let $\pi\simeq\otimes_v'\pi_v$ be an irreducible cuspial automorphic representation of $\PGL_2(\AA)$. 
If $n$ is odd, $\pi_v\simeq D_{2\kap_v}$ for $v\in\frkS_\infty$, none of $\pi_\frkp$ is supercuspidal and $\vep\left(\frac{1}{2},\pi\right)=(-1)^{\ell_\pi^-+d(n-1)/2}$, then $\rmm_\cusp(\Ik_n^\psi(\pi))=1$. 
\end{proposition}

\begin{remark}\label{rem:51}  
Note that $\prod_v\eps_n(\pi_v)=(-1)^{\ell_\pi^-+d(n-1)/2}$. 
Since $\pi_\frkp$ is not supercuspidal, there exists $\chi_\frkp\in\Ome(F_\frkp^\times)$ such that $\pi_\frkp$ is equivalent to the unique irreducible subrepresentation of $I(\chi_\frkp^{},\chi_\frkp^{-1})$. 
Then
\[\vep\left(\frac{1}{2},\pi\right)(-1)^{\ell_\pi^-}=(-1)^{\sum_{v\in\frkS_\infty}\kap_v}\prod_\frkp\chi_\frkp(-1)\]
by Lemma \ref{lem:11}.  
In particular, for $\xi\in F^\times$
\beq 
\vep\left(\frac{1}{2},\pi\otimes\hat\chi^\xi\right)(-1)^{\ell_{\pi\otimes\hat\chi^\xi}^-}=\vep\left(\frac{1}{2},\pi\right)(-1)^{\ell_\pi^-}\prod_{v\in\frkS_\infty}\sgn_v(\xi). \label{tag:51}
\eeq
\end{remark}

\begin{proof}
Fix a totally positive element $\eta$ in $F$. 
Put $\mu_\frkp=\chi_\frkp\hat\chi^{(-1)^{(n-1)/2}\eta}_\frkp$. 
Then $(-1)^{\sum_{v\in\frkS_\infty}\kap_v}\prod_\frkp\mu_\frkp(-1)=1$ by assumption. 
Take $\xi\in F^\times$ so that $\chi_\frkp=\hat\chi^\xi_\frkp\alp_\frkp^{1/2}$ for $\frkp\in\frkS_\pi$. 
Since 
\[\Ik^\psi_n(\pi)\simeq(\otimes_{v\in\frkS_\infty}\frkD^{(n)}_{(2\kap_v+n)/2})\otimes(\otimes_{\frkp\in\frkS_\pi}\St_n^{\psi_\frkp^\xi})\otimes(\otimes'_{\frkp\notin\frkS_\pi}I^{\psi_\frkp}_n(\chi_\frkp))\]
by Lemma \ref{lem:44}, 
Theorem 1.2 and Corollary 8.10 of \cite{IY} applied to $\pi\otimes\hat\chi^{(-1)^{(n-1)/2}\eta}$ say that 
$\rmm_\cusp(\Ik^\psi_n(\pi))=1$. 
\end{proof}

We will construct CAP representations of orthogonal groups as theta lifts of $\Ik^\psi_n(\pi)$. 
Then we realize the other automorphic representations in $\rmMp_n^\psi(\pi)$ as theta lifts of these CAP representations. 
If $\prod_v\vep_v=1$, then there is a quadratic space $V_n^\vep$ of dimension $2n+1$ and discriminant 1 such that $V_n^\vep(F_v)\simeq V_n^{\vep_v}$ for all $v$ by the Minkowski-Hasse theorem (cf. Theorem 4.4 of \cite{Sh5}). 
Let $\rmSO_n^\vep(\pi)=\otimes'_v\rmSO_n^{\vep_v}(\pi_v)$ be a representation of $\SO(V^\vep_n,\AA)$. 

\begin{theorem}\label{thm:51}
Let $n$ be odd and $\pi\simeq\otimes_v'\pi_v$ an irreducible cuspial automorphic representation of $\PGL_2(\AA)$ such that $\pi_v\simeq D_{2\kap_v}$ with $\kap_v>\frac{n}{2}$ for $v\in\frkS_\infty$ and such that none of $\pi_\frkp$ is supercuspidal. 
Then $\rmm_\cusp(\rmSO_n^\vep(\pi))=1$. 
\end{theorem}

\begin{remark}\label{rem:52}
Put $\frkS_\vep=\{v\;|\;\vep_v=-\}$. 
The representation $\rmSO_1^\vep(\pi)$ is the Jacquet-Langlands lift of $\pi$ to $\mathrm{P}D_\vep^\times$, where $D_\vep$ is a quaternion algebra over $F$ which is ramified precisely at places in $\frkS_\vep$ (cf. Proposition \ref{prop:41}).  
\end{remark}

\begin{proof}
Let $x_n^\vep=\prod_vx^{\vep_v}_n:\SO(V_n^\vep,\AA)\to\AA^\times/\AA^{\times 2}$ stand for the spinor norm. 
Since $\rmSO_n^\vep(\pi)\simeq\rmSO_n^\vep(\pi\otimes\hat\chi^\xi)\otimes\hat\chi^\xi\circ x_n^\vep$ by Lemmas \ref{lem:31} and \ref{lem:32}, we see that 
\beq
\rmm_\cusp(\rmSO_n^\vep(\pi))=\rmm_\cusp(\rmSO_n^\vep(\pi\otimes\hat\chi^\xi)) \label{tag:52}
\eeq
for all $\xi\in F^\times$. 
In light of Remark \ref{rem:51} we may impose the condition $\vep\left(\frac{1}{2},\pi\right)=\prod_v\eps_n(\pi_v)$ at the cost of replacing $\pi$ with $\pi\otimes\hat\chi^\xi$ for a suitably chosen $\xi\in F^\times$. 
Then $\rmm_\cusp(\Ik^\psi_n(\pi))=1$ by Proposition \ref{prop:51}.

Take an element $\eta\in F^\times$ such that $(-1)^{(n-1)/2}\vep_v\eta$ is positive in $F_v$ for $v\in\frkS_\infty$ and such that $\frkS_{\pi\otimes\hat\chi^\eta}^-=\{\frkp\;|\;\vep_\frkp=-\}$. 
Then 
\[\vep\left(\frac{1}{2},\pi\otimes\hat\chi^\eta\right)
=(-1)^{\ell_\pi^-+\ell_{\pi\otimes\hat\chi^\eta}^-}\vep\left(\frac{1}{2},\pi\right)\prod_{v\in\frkS_\infty}(-1)^{(n-1)/2}\vep_v=1\]
by (\ref{tag:51}). 
We may further assume that $L\bigl(\frac{1}{2},\pi\otimes\hat\chi^\eta\bigl)\neq 0$ on account of Theorem 4 of \cite{W2}. 
The complete $L$-function 
\[L(s,\pi\otimes\hat\chi^\eta)=L_\bff(s,\pi\otimes\hat\chi^\eta)\prod_{v\in\frkS_\infty}2(2\pi)^{-s-(2\kap_v-1)/2}\vGm\left(s+\kap_v-\frac{1}{2}\right) \] 
is entire and has no zeros in the right half-plane $\Re s\geq\frac{3}{2}$. 
It has no zeros in the left half-plane $\Re s\leq -\frac{1}{2}$ by its functional equation. 
Thus $L_\bff(s,\pi\otimes\hat\chi^\eta)$ is nonzero at $s=1+\frac{n}{2}-j$ for $j=1,2,\dots,n$ by the assumption on $\kap$. 
 
The finite part of the $L$-function of $\Ik^\psi_n(\pi)$ with respect to $\psi^\eta$ is  
\[L_{\psi_\bff^\eta}(s,\Ik^\psi_n(\pi))=\prod_{j=1}^n L_\bff\left(s+\frac{n+1}{2}-j,\pi\otimes\hat\chi^\eta\right) \]
(cf. (\ref{tag:21}) and Lemma \ref{lem:31}). 
It is entire and has no zero at $s=\frac{1}{2}$. 
Since $\frkD_{(2\kap_v+n)/2}^{(n)}$ and $\bar\frkD_{(2\kap_v+n)/2}^{(n)}$ are discrete series representations under the condition that $\kap_v>\frac{n}{2}$, their $L$-factors are holomorphic for $\Re s>0$ by Lemma 7.2 of \cite{Y3}. 
Therefore the complete $L$-function of $\Ik^\psi_n(\pi)$ is holomorphic for $\Re s>0$. 
It is entire by the functional equation. 
Since it has no zero at $s=\frac{1}{2}$, Proposition \ref{prop:22}(\ref{prop:221}) gives a quadratic space $V$ of dimension $2n+1$ and discriminant $1$ such that $\tht^{\psi^\eta}_V(\Ik^\psi_n(\pi))$ is nonzero and cuspidal. 
Because of the choice of $\eta$, Lemmas \ref{lem:31}, \ref{lem:32} and Proposition \ref{prop:41} show that
\begin{align*}
V&\simeq V_n^\vep, & 
\tht^{\psi^\eta}_{V_n^\vep}(\Ik^\psi_n(\pi))&\simeq\rmSO_n^\vep(\pi\otimes\hat\chi^\eta). 
\end{align*}
We finally see that 
\[\rmm_\cusp(\rmSO_n^\vep(\pi))=\rmm_\cusp(\tht^{\psi^\eta}_{V_n^\vep}(\Ik^\psi_n(\pi)))=\rmm_\cusp(\Ik^\psi_n(\pi))=1 \]
by (\ref{tag:52}) and the multiplicity preservation stated in Proposition \ref{prop:22}. 
\end{proof}

\begin{theorem}\label{thm:53}
If $\kap_v>\frac{n}{2}$ for all $v\in\frkS_\infty$, then Conjecture \ref{coj:11} is equivalent to Conjecture \ref{coj:12}. 
\end{theorem}

\begin{proof}
We shall derive Conjecture \ref{coj:11} from Conjecture \ref{coj:12}. 
The opposite implication can be proved by a similar reasoning as in the proof of Theorem \ref{thm:51}. 
Theorem A.2 of \cite{PS} gives $\eta\in F^\times$ such that $L\left(\frac{1}{2},\pi\otimes\hat\chi^\eta\right)\neq 0$. 
Let $\vep:\frkS_\infty\cup\frkS_\pi\to\mu_2$ be such that $\prod_v\vep_v=1$. 
By Conjecture \ref{coj:12} and (\ref{tag:52}) $\rmSO^\vep_n(\pi\otimes\hat\chi^\eta)$ is a cuspidal automorphic representation of $\SO(V_n^\vep,\AA)$. 

In the proof of Theorem \ref{thm:51} we have seen that $L(s,\rmSO^\vep_n(\pi\otimes\hat\chi^\eta))$ is entire and has no zero at $s=\frac{1}{2}$. 
We can therefore apply Proposition \ref{prop:22}(\ref{prop:222}) to see that $\tht^{\psi^\eta}_n(\rmSO^\vep_n(\pi\otimes\hat\chi^\eta))$ occurs in $\scra_\cusp(\Mp(W_n))$ with multiplicity one. 
It belongs to $\rmMp_n^\psi(\pi)$ by Lemmas \ref{lem:31} and \ref{lem:32}. 
Namely, we can write $\tht^{\psi^\eta}_n(\rmSO^\vep_n(\pi\otimes\hat\chi^\eta))=\rmMp_n^\psi(\pi^{\hat\vep})$ with a function $\hat\vep:\frkS_\infty\cup\frkS_\pi\to\mu_2$. 
Proposition \ref{prop:21} implies that if $\vep\neq\vep'$, then $\hat\vep\neq\hat\vep'$. 
We have thus produced $2^{d+\ell_\pi-1}$ mutually nonequivalent irreducible cuspidal automorphic representations of $\Mp(W_n)_\AA$. 
Lemma \ref{lem:51} forces $\hat\vep$ to satisfy $\vep\left(\frac{1}{2},\pi\right)=\prod_v\hat\vep_v$. 
\end{proof}

The following result is derived as a corollary from Theorems \ref{thm:51} and \ref{thm:53}. 

\begin{theorem}\label{thm:52}
Let $\pi$ be an irreducible cuspial automorphic representation of $\PGL_2(\AA)$ such that $\pi_v\simeq D_{2\kap_v}$ with $\kap_v>\frac{n}{2}$ for $v\in\frkS_\infty$. 
If $n$ is odd and none of $\pi_\frkp$ is supercuspidal, 
 then $\rmm_\cusp(\rmMp_n^\psi(\pi^\eps))=\frac{1}{2}\bigl(1+\vep\left(\frac{1}{2},\pi\right)\prod_v\eps_v\bigl)$. 
\end{theorem}


\part{CAP representations for symplectic groups}\label{part:2}

\section{Theta correspondence revisited}\label{sec:6}

Let $F$ be a finite extension of $\QQ_p$. 
When $\sig$ is an irreducible admissible representation of $\O(V)$, we write $\bfn(\sig)$ for the smallest integer $k$ such that $\tht^\psi_k(\sig)$ is nonzero. 
Recall the conservation relation conjectured by Kudla and Rallis \cite{KR2} and proved by Sun and Zhu \cite{SZ}. 

\begin{proposition}[Sun-Zhu \cite{SZ}]\label{prop:61}
For every irreducible admissible representation $\sig$ of $\O(V)$ the following identity holds:  
\[\bfn(\sig)+\bfn(\sig\otimes\sgn_{\O(V)})=\dim V. \]
\end{proposition}

Let $\calv^-_2$ be the anisotropic quadratic space of dimension $4$. 
Put 
\begin{align*}
\calv_n^+&=\scrh_n, & 
\calv_n^-&=\calv^-_2\oplus\scrh_{n-2}
\end{align*}
for $n\geq 2$. 
We denote the Clifford group of $\calv^\vep_n$ by $G(\calv^\vep_n)$. 
Theorem 3.6(\roman{thr}) of \cite{Sh4} gives an isomorphism $\tau^\vep_n$ of $G(\calv^\vep_n)/F^\times$ onto $\O(\calv^\vep_n)$. 
Given $g\in\O(\calv^\vep_n)$, take an element $h\in G(\calv^\vep_n)$ so that $\tau^\vep_n(h)=g$. 
We denote by $x^\vep_n(g)$ the coset represented by $\nu^\vep_n(h)$ in $F^\times/F^{\times 2}$, where the homomorphism $\nu^\vep_n:G(\calv^\vep_n)\to F^\times$ is defined in (3.4) of \cite{Sh4}. 
Clearly $g\mapsto x^\vep_n(g)$ gives a homomorphism of $\O(\calv^\vep_n)\to F^\times/F^{\times 2}$. 

We shall describe some basic properties of theta correspondence for similitudes and relate it to the usual theta correspondences for isometric groups. 
The symplectic similitude group of the symplectic space $W_k$ is defined by 
\[\GSp(W_k)=\GSp_k=\{g\in\GL_{2k}\;|\;gJ_k\trs g=\lam_k(g)J_k\}, \]
where $\lam_k(g)$ is a scalar. 
Let $\lam_n^\pm$ be the similitude factor of $\GO(\calv_n^\pm)$.  
As is well-known, $\lam_n^\pm(\GO(\calv_n^\pm))=F^\times$. 
When $\vPi$ is a representation of $Sp_n$ (resp. $\O(\calv_n^\pm)$) and $h\in\GSp_n$ (resp. $\GO(\calv_n^\pm)$), 
let $\vPi^h$ denote the equivalence class of the representation $g\mapsto\vPi(h^{-1}gh)$. 
 
We shall consider the product group $R=\GO(\calv_n^\pm)\times\GSp(W_k)$. 
This group contains the subgroup 
\[R_0=\{(h,g)\in R\;|\;\lam^\pm_n(h)\lam_k(g)=1\}. \]
The Weil representation $\ome^\psi_{\calv_n^\pm}$ naturally extends to the group $R_0$. 
Now consider the compactly induced representation $\ome_{\calv_n^\eps}=\cind^R_{R_0}\ome^\psi_{\calv_n^\eps}$, which is independent of $\psi$ as a representation of $R$. 
For any irreducible representation $\sig$ of $\GO(\calv_n^\eps)$ (resp. $\GSp(W_k)$) the maximal $\sig$-isotypic quotient of $\ome_{\calv_n^\eps}$ has the form $\sig\boxtimes\vTh(\sig)$, where $\vTh(\sig)$ is a representation of $\GSp(W_k)$ (resp. $\GO(\calv_n^\eps)$). 
Further, we let $\tht(\sig)$ be the maximal semisimple quotient of $\vTh(\sig)$. 
Then $\tht(\sig)$ is irreducible whenever $\vTh(\sig)$ is nonzero, and the map $\sig\mapsto\tht(\sig)$ is injective on its domain by the extended Howe conjecture for similitudes, which follows from the Howe conjecture for isometry groups. 
The following lemma relates the theta correspondence for isometries and similitudes. 

\begin{lemma}\label{lem:61}
\begin{enumerate}
\renewcommand\labelenumi{(\theenumi)}
\item\label{lem:611} Let $\sig$ be an irreducible representation of a similitude group and $\tau$ a constituent of the restriction of $\sig$ to the relevant isometry group. 
If $\tht^\psi(\tau)$ is nonzero, then $\tht(\sig)$ is nonzero.
\item\label{lem:612} If $\Hom_R(\ome_{\calv_n^\eps},\sig_1\boxtimes\sig_2)\neq 0$, then there is a bijection $f$ between irreducible summands of $\sig_1|_{\O(\calv^\eps_n)}$ and irreducible summands of $\sig_2|_{Sp(W_n)}$ such that for any irreducible summand $\tau_i$ in the restriction of $\sig_i$ to the relevant isometry group 
\[\tau_2=f(\tau_1)\Leftrightarrow\Hom_{\O(\calv_n^\eps)\times Sp(W_k)}(\ome^\psi_{\calv_n^\eps},\tau_1\boxtimes\tau_2)\neq 0. \]
\end{enumerate} 
\end{lemma}

We work over a number field $F$ in the following proposition. 

\begin{proposition}[\cite{Y3}]\label{prop:62}
\begin{enumerate}
\renewcommand\labelenumi{(\theenumi)}
\item\label{prop:621} Let $\vPi$ be an irreducible cuspidal automorphic representation of $Sp_n(\AA)$ such that $L(s,\vPi)$ is holomorphic for $\Re s>1$ and has a pole at $s=1$. 
Then there is a unique equivalence class of quadratic spaces $\calv$ of dimension $2n$ and discriminant $1$ such that $\tht^\psi_\calv(\vPi)$ is nonzero. 
Moreover, $\tht^\psi_\calv(\vPi)$ is cuspidal and the multiplicity of $\tht^\psi_\calv(\vPi)$ in $\tht^\psi_\calv(\scra_\cusp(Sp_n))$ is equal to $\rmm_\cusp(\vPi)$.  
\item\label{prop:622} Let $\calv$ be a quadratic space of dimension $2n$ and discriminant $1$ and $\sig$ an irreducible cuspidal automorphic representation of $\O(\calv,\AA)$ such that $L(s,\sig)$ is entire and $L(1,\sig)\neq 0$. 
Then $\tht^\psi_n(\sig)$ is cuspidal. 
Moreover, $\tht^\psi_n(\sig)$ is nonzero if and only if $\tht^{\psi_v}_n(\sig_v)$ is nonzero for all $v$. 
The multiplicity of $\tht^\psi_n(\sig)$ in $\tht^\psi_n(\scra_\cusp(\O(\calv)))$ is equal to $\rmm_\cusp(\sig)$. 
\end{enumerate}
\end{proposition}

\begin{proof}
As in the proof of Proposition \ref{prop:22}, one can deduce these results from \cite[Theorem 10.1, Lemma 10.2]{Y3} and \cite[Proposition 2.14]{G}. 
\end{proof}

\section{Saito-Kurokawa representations}\label{sec:7}

Let $F$ be a $\frkp$-adic field. 
We begin by recalling the definition and construction of the Saito-Kurokawa $A$-packet $\SK_2(\pi)$ on $\PGSp_2$ associated to an irreducible admissible unitary generic representation $\pi$ of $\PGL_2(F)$. 
Recall the quadratic space $V_2^+=V^+_1\oplus\scrh_1$ of dimension $5$ and the isomorphism $\PGSp_2\simeq\SO(V^+_2)$. 
We consider the theta correspondence for the dual pair $\Mp(W_1)\times\PGSp_2$ to construct $\SK_2(\pi)$, which is independent of $\psi$, by 
\begin{align*}
\SK_2(\pi)&=\{\SK_2^+(\pi),\;\SK_2^-(\pi)\}, & 
\SK_2^\pm(\pi)&=\tht^\psi_{V^+_2}(\rmMp^\psi_1(\pi^\pm)), & 
\end{align*}

Schmidt \cite{Sc} gives another way of constructing the packet $\SK_2(\pi)$. 
We employ the symbols $D_+$ and $D_-$ defined in Section \ref{sec:3}, which are the split and division quaternion algebras over $F$. 
Denote the reduced norm of $D_\vep$ by $\mathrm{N}_\vep$. 
Then $\calv_2^\vep\simeq (D_\vep,\mathrm{N}_\vep)$. 
We have an isomorphism
\[\GSO(\calv^\vep_2)\simeq D_\vep^\times\times D_\vep^\times/\{(z,z^{-1})\;|\;z\in F^\times\}, \]
via the action of the latter on $D_\vep$ given by $(\alp,\bet)\mapsto \alp x\bar\bet$. 
Moreover, an element of $\GO(\calv^\vep_2)$ of determinant $-1$ is given by the conjugation action $c:x\mapsto\bar x$ on $D_\vep$. 

An irreducible representation of $\GSO(\calv^\vep_2)$ is of the form $\sig_1\boxtimes\sig_2$, where $\sig_i$ are representations of $D_\vep^\times$ with equal central character. 
Moreover, the action of $c$ on representations of $\GSO(\calv^\vep_2)$ is given by $\sig_1\boxtimes\sig_2\mapsto\sig_2\boxtimes\sig_1$.
Given a representation $\pi$ of $\PGL_2(F)$, let 
\[\rmGO^\vep_2(\pi)=\ind^{\GO(\calv_2^\vep)}_{\GSO(\calv_2^\vep)}\pi^\vep\boxtimes\1_{D_\vep^\times}. \]
The representation $\rmGO^\vep_2(\pi)$ is irreducible unless $\vep=-$ and $\pi\simeq\St$. 

\begin{proposition}[Schmidt \cite{Sc}]\label{prop:71} 
\begin{enumerate}
\renewcommand\labelenumi{(\theenumi)}
\item\label{prop:711} $\SK_2^\vep(\pi)=\tht(\rmGO^\vep_2(\pi))$ unless $\vep=-$ and $\pi\simeq\St$. 
\item\label{prop:712} $\SK_2^-(\St)=\tht(\1_{\GO(\calv_2^-)})$. 
\end{enumerate}
\end{proposition}


We associate to $\mu\in\Ome(F^\times)$ the degenerate principal series representation 
\[I_n(\mu)=\Ind^{Sp_n}_{P_n^n}\mu\circ\det. \]
The following results are included in the paper \cite{KR} by Kudla and Rallis. 

\begin{lemma}[\cite{KR}]\label{lem:71}
Assume that $n$ is even. 
Let $\mu\in\Ome(F^\times)$. 
\begin{enumerate}
\renewcommand\labelenumi{(\theenumi)}
\item\label{lem:711} If $-\frac{1}{2}<\Re\mu<\frac{1}{2}$, then $I_n(\mu)$ is irreducible. 
\item\label{lem:712} If $\chi^2=1$, then $I_n(\chi\alp^{1/2})$ has a unique irreducible subrepresentation, which we denote by $\St_n(\chi)$. 
\end{enumerate}
\end{lemma}

When $\chi$ is the trivial character of $F^\times$, we shall write $\St_n=\St_n(\chi)$. 

\begin{lemma}[\cite{G,GG}]\label{lem:72}
\begin{enumerate}
\renewcommand\labelenumi{(\theenumi)}
\item\label{lem:721} $\SK_2^+(\pi)|_{Sp_2}$ is equivalent to the Langlands quotient of the standard module $\Ind^{Sp_2}_{P_2^2}(\pi\otimes\alp^{1/2})$. 
\item\label{lem:722} $\SK_2^-(\St)|_{Sp_2}\simeq\St_2\simeq\tht^\psi_2(\1_{\O(\calv_2^-)})$ is tempered. 
\item\label{lem:723} If $\pi$ is supercuspidal or $\pi\simeq\St\otimes\chi$ with $\chi$ a nontrivial quadratic character of $F^\times$, then $\SK_2^-(\pi)$ is supercuspidal. 
\end{enumerate}
\end{lemma}

\begin{proof}
These assertions are stated in Proposition 5.5 of \cite{G} (see the paragraph after proof of \cite[Theorem 8.1]{GG}). 
\end{proof}

We assume $n$ to be even for the remainder of this paper. 
Let $\G\scrq_n$ be a parabolic subgroup of $\GSp_n$ with Levi subgroup 
\[\GL_2(F)\times\cdots\times\GL_2(F)\times\GSp_2\] 
and $\G Q_n^\pm$ a parabolic subgroup of $\GO(\calv_n^\pm)$ with Levi subgroup 
\[\GL_2(F)\times\cdots\times\GL_2(F)\times\GO(\calv_2^\pm). \] 
Let $\scrp_n$ (resp. $P_n^+$) denote a parabolic subgroup of $Sp_n$ (resp. $\O(\calv_n^+)$) whose Levi subgroup is isomorphic to $\GL_2(F)\times\cdots\times\GL_2(F)$.  
Put 
\begin{align*}
\scrq_n&=\G\scrq_n\cap Sp_n, & 
Q_n^\pm&=\G Q_n^\pm\cap \O(\calv_n^\pm). 
\end{align*} 

\begin{definition}\label{def:71}
Put $\bet_k=\frac{k^2-4}{8}$. 
We denote the Langlands quotient of 
\[\Ind^{\GSp_n}_{\G\scrq_n}(\pi\otimes\alp^{(n-1)/2})\boxtimes(\pi\otimes\alp^{(n-3)/2})\boxtimes\cdots\boxtimes(\pi\otimes\alp^{3/2})\boxtimes(\SK_2^\pm(\pi)\otimes\alp^{-\bet_n}\circ\lam_2) \] 
by $\SK^\pm_n(\pi)$ and denote the Langlands quotient of  
\[\Ind^{\GO(\calv_n^\pm)}_{\G Q^\pm_n}(\pi\otimes\alp^{(n-1)/2})\boxtimes(\pi\otimes\alp^{(n-3)/2})\boxtimes\cdots\boxtimes(\pi\otimes\alp^{3/2})\boxtimes(\pi^\pm\boxtimes\1)^\star\otimes\alp^{-\bet_n}\circ\lam_2^\pm \]
by $\rmGO_n^\pm(\pi)$. 
Here $\SK^-_n(\pi)$ and $\rmGO_n^-(\pi)$ are interpreted as $0$ if $\pi^-=0$. 
When these representations are restricted to the isometry groups, we write 
\begin{align*}
\rmSp^\pm_n(\pi)&=\SK^\pm_n(\pi)|_{Sp_n}, & 
\rmO_n^\pm(\pi)&=\rmGO_n^\pm(\pi)|_{\O(\calv^\pm_n)}. 
\end{align*}
\end{definition}


\begin{lemma}[cf. Theorem 8.1 of \cite{GG}]\label{lem:73}
The representations $\rmSp^\eps_n(\pi)$ and $\rmO_n^\eps(\pi)$ are irreducible unless $\pi\simeq\St\otimes\chi$, $\chi^2=1$, $\chi\neq 1$ and $\eps=-$. 
\end{lemma}

\begin{proof}
One can readily deduce the general case from the special case in which $n=2$. 
This case is proved in the proof of Theorem 8.1 of \cite{GG}.  
\end{proof}


\begin{lemma}\label{lem:74}
\begin{enumerate}
\renewcommand\labelenumi{(\theenumi)}
\item\label{lem:741} $\rmSp_n^+(\pi)$ is equivalent to the unique irreducible quotient of $\Ind^{Sp_n}_{\scrp_n}(\pi\otimes\alp^{(n-1)/2})\boxtimes(\pi\otimes\alp^{(n-3)/2})\boxtimes\cdots\boxtimes(\pi\otimes\alp^{1/2})$. 
\item\label{lem:742} If $-\frac{1}{2}<\Re\mu<\frac{1}{2}$, then $\rmSp^+_n(I(\mu,\mu^{-1}))\simeq I_n(\mu)$. 
\item\label{lem:743} If $\chi^2=1$ and $\chi\neq 1$, then $\rmSp^+_n(\St\otimes\chi)\simeq\St_n(\chi)$. 
\item\label{lem:744} $\rmSp^-_n(\St)\simeq \St_n$. 
\end{enumerate}
\end{lemma}

\begin{proof}
Lemma \ref{lem:72}(\ref{lem:721}) implies (\ref{lem:741}). 
Consequently, (\ref{lem:742}) can be seen from \S 6.2 of \cite{IY}.
We obtain (\ref{lem:743}) as a special case of \cite[Proposition 3.11(2)]{J}.  
Proposition 3.10(2) of \cite{J} tells us that $\St_n$ is equivalent to the Langlands quotient of 
\[\Ind^{Sp_n}_{\scrq_n}(\St\otimes\alp^{(n-1)/2})\boxtimes(\St\otimes\alp^{(n-3)/2})\boxtimes\cdots\boxtimes(\St\otimes\alp^{3/2})\boxtimes\St_2. \]
This combined with Lemma \ref{lem:72}(\ref{lem:722}) proves (\ref{lem:744}). 
\end{proof}

Given a quadratic character $\chi:F^\times\to\mu_2$,  we denote the unique irreducible quotient of the standard module 
\[\Ind^{\O(\calv_n^-)}_{Q^-_n}(\St\otimes\chi\alp^{(n-1)/2})\boxtimes(\St\otimes\chi\alp^{(n-3)/2})\boxtimes\cdots\boxtimes(\St\otimes\chi\alp^{3/2})\boxtimes\chi\circ x^-_2 \]
by $\Xi_n^\chi$. 
When $\chi$ is trivial, we will write $\Xi_n=\Xi_n^\chi$. 
Let $Q_n^n$ denote the parabolic subgroup of $\O(\calv_n^+)$ stabilizing a maximal totally isotropic subspace of $\calv_n^+$ and define the degenerate principal series representation $J_n(\mu)=\Ind^{\O(\calv_n^+)}_{Q_n^n}\mu$.

\begin{lemma}\label{lem:75}
\begin{enumerate}
\renewcommand\labelenumi{(\theenumi)}
\item\label{lem:751} $\rmO_n^+(\pi)$ is equivalent to the unique irreducible quotient of $\Ind^{\O(\calv_n^+)}_{P^+_n}(\pi\otimes\alp^{(n-1)/2})\boxtimes(\pi\otimes\alp^{(n-3)/2})\boxtimes\cdots\boxtimes(\pi\otimes\alp^{1/2})$. 
\item\label{lem:752} $\rmO_n^+(\pi)\simeq\rmO_n^+(\pi)\otimes\sgn_{\O(\calv_n^+)}$. 
\item\label{lem:753} If $\chi^2=1$, then $\rmO_n^+(\pi)\otimes\chi\circ x^+_n\simeq\rmO_n^+(\pi\otimes\chi)$ and $\rmO_n^-(\St)\otimes\chi\circ x^-_n\simeq\Xi_n^\chi$. 
\item\label{lem:754} If $\chi^2=1$ and $\chi\neq 1$, then $\rmO_n^-(\St\otimes\chi)\simeq\Xi_n^\chi\oplus\Xi_n^\chi\otimes\sgn_{\O(\calv_n^-)}$. 
\item\label{lem:755} If $-\frac{1}{2}<\Re\mu<\frac{1}{2}$, then $\O_n^+(I(\mu,\mu^{-1}))\simeq J_n(\mu)$. 
\item\label{lem:756} If $\chi^2=1$, then $\O_n^+(\St\otimes\chi)$ is equivalent to the unique irreducible subrepresentation of $J_n(\chi\alp^{1/2})$. 
\end{enumerate}
\end{lemma}

\begin{proof}
Lemma 6.1 of \cite{GT} implies (\ref{lem:751}), from which (\ref{lem:752}) is clear.  
Note that 
\[\SO(\calv^\vep_2)\simeq \{(\alp,\bet)\in D_\vep^\times\times D_\vep^\times\;|\;\mathrm{N}_\vep(\alp)=\mathrm{N}_\vep(\bet)^{-1}\}/\{(z,z^{-1})\;|\;z\in F^\times\}. \]
One can observe from \S 7.4(B) of \cite{Sh4} that 
\begin{align*}
x^\vep_2((\alp,\bet))&=\mathrm{N}_\vep(\alp)F^{\times 2}=\mathrm{N}_\vep(\bet)F^{\times2}, & 
x^\vep_2(c)&=F^{\times2}. 
\end{align*}
Since $(\1\boxtimes\1)^\star=\1_{\GO(\calv_2^-)}$, we can prove (\ref{lem:753}). 
If $\chi$ has order $2$, then 
\[\O_2^-(\St\otimes\chi)\simeq \Ind^{\O(\calv^-_2)}_{\SO(\calv^-_2)}\chi\circ\mathrm{N}_-\boxtimes\1\simeq\chi\circ x^-_2\oplus(\chi\circ x^-_2\otimes\sgn_{\O(\calv_2^-)}), \] 
which implies (\ref{lem:754}). 
The proof for (\ref{lem:755}) is similar to the proof for Lemma \ref{lem:74}(\ref{lem:742}). 
The last statement (\ref{lem:756}) is proved in \cite{BJ} (see Theorem 8.2 of \cite{Y2}). 
\end{proof}



The following result is Proposition C.4(\roman{two}) of \cite{GI}. 

\begin{proposition}[\cite{GI}]\label{prop:81}
Let $Q$ be a parabolic subgroup of $\O(\calv^\eps_n)$ with Levi component 
\begin{align*}
&\GL_{n_1}(F)\times\cdots\times\GL_{n_r}(F)\times\O(\calv^\eps_{n_0}), & 
n&=n_0+n_1+\cdots+n_r.  
\end{align*}
\begin{enumerate}
\renewcommand\labelenumi{(\theenumi)}
\item\label{prop:811} Let $\sig$ be a Langlands quotient of a standard module 
\[\Ind^{\O(\calv_n^\eps)}_Q(\sig_1\otimes\alp^{s_1})\boxtimes\cdots\boxtimes(\sig_r\otimes\alp^{s_r})\boxtimes\pi_0, \]
where $\pi_0$ is a tempered representation of $\O(\calv_{n_0}^\eps)$, $\sig_i$ is a tempered representation of $\GL_{n_i}(F)$ and $s_1>\cdots>s_r>0$. 
If $\tht_n(\sig)$ is nonzero, then it is a quotient of 
\[\Ind^{Sp_n}_P(\sig_1\otimes\alp^{s_1})\boxtimes\cdots\boxtimes(\sig_r\otimes\alp^{s_r})\boxtimes\Tht_{n_0}^\psi(\pi_0), \]
where $\G P$ is a parabolic subgroup of $\GSp_n$ with Levi component 
\[\GL_{n_1}(F)\times\cdots\times\GL_{n_r}(F)\times Sp_{n_0}(F). \]
\item\label{prop:812} Let $\vPi$ be a Langlands quotient of a standard module 
\[\Ind^{Sp_{n-1}}_{P'}(\sig_1\otimes\alp^{s_1})\boxtimes\cdots\boxtimes(\sig_r\otimes\alp^{s_r})\boxtimes\rho_0, \]
where $P'$ is a parabolic subgroup of $Sp_{n-1}(F)$ with Levi component 
\begin{align*}
&\GL_{n_1}(F)\times\cdots\times\GL_{n_r}(F)\times Sp_{n_0-1}(F), & 
n&=n_0+n_1+\cdots+n_r, & 
n_0\geq 1, 
\end{align*}
$\rho_0$ is a tempered representation of $Sp_{n_0-1}(F)$, $\sig_i$ is a tempered representation of $\GL_{n_i}(F)$ and $s_1>\cdots>s_r>0$. 
If $\tht_{\calv_n^\eps}(\vPi)$ is nonzero, then it is a quotient of 
\[\Ind^{\O(\calv_n^\eps)}_Q(\sig_1\otimes\alp^{s_1})\boxtimes\cdots\boxtimes(\sig_r\otimes\alp^{s_r})\boxtimes\Tht^\psi_{\calv^\eps_{n_0}}(\rho_0). \]
\end{enumerate}
\end{proposition}

In the following lemma we extend the theta correspondence to semisimple representations. 

\begin{lemma}\label{lem:81}
If $\pi$ is an irreducible admissible unitary infinite dimensional representation of $\PGL_2(F)$, then $\tht_n^\psi(\rmGO_n^\pm(\pi))\simeq \SK^\pm_n(\pi)$. 
\end{lemma}

\begin{proof}
Once we show that $\tht_n^\psi(\rmO_n^\pm(\pi))$ is nonzero, 
Proposition \ref{prop:81}(\ref{prop:811}) together with Lemmas \ref{lem:72} and \ref{lem:75}(\ref{lem:751}) prove $\tht_n^\psi(\rmO_n^\pm(\pi))\simeq \rmSp^\pm_n(\pi)$. 
Proposition \ref{prop:61} combined with Lemma \ref{lem:75}(\ref{lem:752}) shows that $\bfn(\rmO_n^+(\pi))=n$. 
In particular, $\tht^\psi_n(\rmO_n^+(\pi))$ is nonzero. 

Assume that $\pi$ is a discrete series representation. 
Let $\sig$ be an irreducible summand of $\rmO_n^-(\pi)$. 
Put $\sig'=\sig\otimes\sgn_{\O(\calv_n^-)}$. 
Seeking a contradiction, we suppose that $\tht_n^\psi(\sig)$ is zero. 
Then $\tht_{n-1}^\psi(\sig')$ is nonzero by Proposition \ref{prop:61}. 
Put $\vPi=\tht_{n-1}^\psi(\sig')$.  
If we write $\vPi$ as in Proposition \ref{prop:81}(\ref{prop:812}), then since $\tht_{\calv_n^-}^\psi(\vPi)\simeq\sig'$ by symmetry, $r=\frac{n-2}{2}$, $n_1=n_2=\cdots=n_r=2$ and $\pi_0$ is a tempered representation of $\SL_2(F)$ whose theta lift is equivalent to an irreducible summand of $(\pi^-\boxtimes\1)^\star|_{\O(\calv_2^-)}$. 
If $\pi^-\not\simeq\1$, then the theta lift of $(\pi^-\boxtimes\1)^\star$ to $\GL_2(F)$ is zero by Lemma 4.1(\roman{one}) of \cite{GG}, which is a contradiction in view of Lemma \ref{lem:61}(\ref{lem:611}). 
Thus $\pi^-\simeq\1$ and hence $\pi\simeq\St$ and $\sig'$ is equivalent to the unique irreducible subrepresentation of
\[\Ind^{\O(\calv_n^-)}_{Q^-_n}(\pi\otimes\alp^{(n-1)/2})\boxtimes(\pi\otimes\alp^{(n-3)/2})\boxtimes\cdots\boxtimes(\pi\otimes\alp^{3/2})\boxtimes\sgn_{\O(\calv_2^-)}. \]
Since $\bfn(\sgn_{\O(\calv_2^-)})=4$, Proposition \ref{prop:81}(\ref{prop:811}) again implies that $\tht^\psi_{n+1}(\sig')$ is zero. 
A fortiori, $\tht^\psi_{n-1}(\sig')$ is zero, which is a contradiction. 
\end{proof}

Finally we discuss the real case.  
Let $n$ be an even natural number. 
The direct sum $\frkD^{(n)}_k\oplus\bar\frkD^{(n)}_k$ define a representation of $\PGSp_n(\RR)$. 
 For $\kap\in\NN$ we define the representation $\SK_n^{(-1)^{n/2}}(D_{2\kap})$ of $\PGSp_n(\RR)$ by 
\[\SK_n^{(-1)^{n/2}}(D_{2\kap})=\frkD^{(n)}_{(2\kap+n)/2}\oplus\bar\frkD^{(n)}_{(2\kap+n)/2}. \]
Let $\calv^{(-1)^{n/2}}_n$ be a positive definite quadratic space of dimension $2n$. 
For $\kap\geq\frac{n}{2}$ let $\rmO_n^{(-1)^{n/2}}(D_{2\kap})$ be the irreducible representation of $\O\big(\calv_n^{(-1)^{n/2}}\big)$ whose highest weight is 
\beq
\left(\kap-\frac{n}{2},\kap-\frac{n}{2},\dots,\kap-\frac{n}{2},0,\dots,0\right), \label{tag:81}
\eeq
where $\kap-\frac{n}{2}$ occurs $n$ times. 

\section{CAP representations for even orthogonal groups}\label{sec:9}

We go back to the global setting. 
Our basic number field $F$ is totally real. 
Let $n$ be an even natural number and $\pi\simeq\otimes_v'\pi^{}_v$ an irreducible cuspidal automorphic representation of $\PGL_2(\AA)$ such that $\pi_v\simeq D_{2\kap_v}$ for $v\in\frkS_\infty$. 
When $\eps:\frkS_\infty\cup\frkS_\pi\to\mu_2$ satisfies $\eps_v=(-1)^{n/2}$ for $v\in\frkS_\infty$, we let $\SK^\eps_n(\pi)=\otimes'_v \SK^{\eps_v}_n(\pi_v)$ be an admissible representation of $\PGSp_n(\AA)$. 

\begin{conjecture}\label{coj:91}
$\rmm_\cusp(\SK_n^\eps(\pi))=\frac{1}{2}\bigl(1+\vep\left(\frac{1}{2},\pi\right)\prod_v\eps_v\bigl)$.  
\end{conjecture}

When $n=2$, Conjecture \ref{coj:91} is the well-known result of Piatetski-Shapiro \cite{PS2} (cf. \cite{G}). 
Though similitude groups may be more natural, we shall consider the isometry groups due to the lack of suitable extensions of Lemma \ref{lem:75}(\ref{lem:753}), (\ref{lem:754}) to $\GO(\calv_n^\pm)$. 
Lemma \ref{lem:73} shows that $Sp^\eps_n(\pi)=\otimes'_v Sp^{\eps_v}_n(\pi_v)$ is the sum of $2^{d+k}$ irreducible representations as an abstract representation of $Sp_n(\AA)$, where $k$ is the number of finite primes $\frkp$ such that $\frkp\in\frkS_\pi\setminus\frkS_\pi^-$ and such that $\eps_\frkp=-$.
Recall the sign $\eps_n(\pi_v)\in\mu_2$ defined in Remark \ref{rem:11}. 

\begin{proposition}\label{prop:91}
Let $\pi\simeq\otimes_v'\pi_v$ be an irreducible cuspial automorphic representation of $\PGL_2(\AA)$ such that $\pi_v\simeq D_{2\kap_v}$ for every $v\in\frkS_\infty$ and such that none of $\pi_\frkp$ is supercuspidal. 
Let 
\beq
\Ik_n(\pi)=(\otimes_{v\in\frkS_\infty}\frkD^{(n)}_{(2\kap_v+n)/2})\otimes(\otimes'_\frkp \rmSp^{\eps_n(\pi_\frkp)}_n(\pi_\frkp)) \label{tag:91}
\eeq
be an irreducible representation of $Sp_n(\AA)$. 
If $\vep\left(\frac{1}{2},\pi\right)=(-1)^{\ell_\pi^-+dn/2}$, then 
\[\rmm_\cusp(\Ik_n(\pi))=1. \]
\end{proposition}

\begin{proof}
We retain the notation of Remark \ref{rem:51}. 
Take $\xi\in F^\times$ so that $\chi_\frkp=\hat\chi^\xi_\frkp\alp_\frkp^{1/2}$ for $\frkp\in\frkS_\pi$. 
Lemma \ref{lem:74} shows that $\Ik_n(\pi)$ is equivalent to 
\[(\otimes_{v\in\frkS_\infty}\frkD^{(n)}_{(2\kap_v+n)/2})\otimes(\otimes_{\frkp\in\frkS_\pi}\St_n(\hat\chi_\frkp^\xi))\otimes(\otimes'_{\frkp\notin\frkS_\pi}I_n(\chi_\frkp)). \]
Fix a totally positive element $\eta$ in $F$. 
Put $\mu_\frkp=\chi_\frkp\hat\chi^{(-1)^{n/2}\eta}_\frkp$. 
By assumption $(-1)^{\sum_{v\in\frkS_\infty}\kap_v}\prod_\frkp\mu_\frkp(-1)=1$. 
We can now apply \cite[Theorem 1.2, Remark 6.2(1)]{IY} to $\pi\otimes\hat\chi^{(-1)^{n/2}\eta}$ to observe that $\rmm_\cusp(\Ik_n(\pi))=1$. 
\end{proof}

Let $\vep:\frkS_\infty\cup\frkS_\pi\to\mu_2$ be a function which satisfies $\prod_v\vep_v=1$ and $\vep_v=(-1)^{n/2}$ for $v\in\frkS_\infty$. 
 Theorem 4.4 of \cite{Sh5} gives rise to a quadratic space $\calv_n^\vep$ over $F$ whose localization at $v$ is isometric to $\calv_n^{\vep_v}$ for all $v$. 
Let $\rmO_n^\vep(\pi)=\otimes'_v\rmO_n^{\vep_v}(\pi_v)$ be a representation of $\O(\calv^\vep_n,\AA)$. 

\begin{theorem}\label{thm:91}
Let $\pi$ be an irreducible cuspial automorphic representation of $\PGL_2(\AA)$ such that $\pi_v\simeq D_{2\kap_v}$ with $\kap_v\geq\frac{n}{2}$ for $v\in\frkS_\infty$ and such that none of $\pi_\frkp$ is supercuspidal. 
If $\prod_v\vep_v=1$ and $\vep_v=(-1)^{n/2}$ for $v\in\frkS_\infty$, then every irreducible summand of $\rmO_n^\vep(\pi)$ occurs in $\scra_\cusp(\O(\calv_n^\vep))$ with multiplicity one. 
\end{theorem}


\begin{proof}
For $\frkp\in\frkS_\pi$ there is a quadratic character $\chi_\frkp$ of $F_\frkp^\times$ such that $\pi_\frkp\simeq\St_\frkp\otimes\chi_\frkp$ by assumption. 
Put 
\begin{align*}
\frkS_\vep&=\{\frkp\;|\;\vep_\frkp=-\}, & 
\frkS_\pi^\vep&=\frkS_\pi\cap\frkS_\vep. 
\end{align*}
Let 
\[\Xi_n^\vep(\pi)=(\otimes_{\frkp\in\frkS_\pi^\vep}\Xi_n^{\chi_\frkp})\otimes(\otimes_{v\notin\frkS_\pi^\vep}\rmO_n^\vep(\pi_v))\]
be an irreducible representation of $\O(\calv_n^\vep,\AA)$. 
Given a finite set $T$ of places of $F$ of even cardinality, we can define the automorphic character $\sgn_T:\O(\calv_n^\vep,\AA)\to\mu_2$ by $\sgn_T(g)=\prod_{v\in T}\sgn_{\O(\calv_n^{\vep_v})}(g_v)$. 
When $\sig$ is an irreducible summand of $\rmO_n^\vep(\pi)$, Lemma \ref{lem:75}(\ref{lem:752}), (\ref{lem:754}) gives $T$ such that $\sig\simeq\Xi_n^\vep(\pi)\otimes\sgn_T$. 
Thus $\rmm_\cusp(\sig)=\rmm_\cusp(\Xi_n^\vep(\pi))$. 
We will show that $\rmm_\cusp(\Xi_n^\vep(\pi))=1$. 

Define a homomorphism $x_n^\vep:\O(\calv_n^\vep,\AA)\to\AA^\times/\AA^{\times 2}$ by $x_n^\vep(g)=(x^{\vep_v}_n(g_v))$ for $g=(g_v)\in\O(\calv_n^\vep,\AA)$. 
The composition $\hat\chi^\xi\circ x_n^\vep$ is an automorphic character of $\O(\calv_n^\vep,\AA)$ for every $\xi\in F^\times$. 
If $\kap_v>\frac{n}{2}$, then 
\[\Xi_n^\vep(\pi)\simeq\Xi_n^\vep(\pi\otimes\hat\chi^\xi)\otimes\hat\chi^\xi\circ x_n^\vep \]
for all $\xi$ by (\ref{tag:33}) and Lemma \ref{lem:75}(\ref{lem:753}). 
If $\kap_v=\frac{n}{2}$, then there is $T$ such that 
\[\Xi_n^\vep(\pi)\simeq\Xi_n^\vep(\pi\otimes\hat\chi^\xi)\otimes(\hat\chi^\xi\circ x_n^\vep)\sgn_T. \]
One sees that 
\[\rmm_\cusp(\Xi_n^\vep(\pi))=\rmm_\cusp(\Xi_n^\vep(\pi\otimes\hat\chi^\xi)). \]

In view of Remark \ref{rem:51} we may assume that $\vep\left(\frac{1}{2},\pi\right)=(-1)^{\ell_\pi^-+dn/2}$ at the cost of replacing $\pi$ with $\pi\otimes\hat\chi^\xi$ for a suitably chosen $\xi\in F^\times$. 
Take a totally positive element $\eta\in F^\times$ such that $\frkS_{\pi\otimes\hat\chi^\eta}^-=\frkS_\vep$. 
Put $\sig=\pi\otimes\hat\chi^\eta$. 
Since $(-1)^{\ell_\sig^-}=\prod_\frkp\vep_\frkp=\prod_{v\in\frkS_\infty}\vep_v=(-1)^{dn/2}$, 
\[\vep\left(\frac{1}{2},\pi\otimes\hat\chi^\eta\right)
=(-1)^{\ell_\pi^-+\ell_{\pi\otimes\hat\chi^\eta}^-}\vep\left(\frac{1}{2},\pi\right)
=(-1)^{\ell_\pi^-+dn/2}\vep\left(\frac{1}{2},\pi\right)=1\]
by (\ref{tag:51}). 
Since $\vep\left(\frac{1}{2},\sig\right)=1=\prod_v\vep_v=(-1)^{\ell_\sig^-+dn/2}$, Proposition \ref{prop:91} shows that $\rmm_\cusp(\Ik_n(\sig))=1$.
We may assume that $L\bigl(\frac{1}{2},\sig\bigl)\neq 0$ on account of Theorem 4 of \cite{W2}. 
 
Let $L_\bff(s,\Ik_n(\sig))=\prod_\frkp L(s,\rmSp_n^{\eps_n(\sig_\frkp)}(\sig_\frkp))$ be the finite part of the standard $L$-function of $\Ik_n(\sig)$. 
Lemma \ref{lem:74} says that 
\[\frac{L_\bff(s,\Ik_n(\sig))}{\zet_\bff(s)\prod_{j=1}^n L_\bff\left(s+\frac{n+1}{2}-j,\sig\right)}=\prod_{\frkp\in\frkS_\sig^-}\frac{L(s,\St_{2,\frkp})}{\zet_\frkp(s)L\left(s+\frac{1}{2},\St_\frkp\right)L\left(s-\frac{1}{2},\St_\frkp\right)}. \]
In the proof of Theorem \ref{thm:51} we have seen that $L_\bff(s,\sig)$ is nonzero at $s=\frac{n+1}{2}+1-j$ for $j=1,2,\dots,n$. 
Thus the denominator of the left hand side has a pole at $s=1$. 
Since $L(s,\St_{2,\frkp})$ is holomorphic for $\Re s>0$ by Lemma \ref{lem:72}(\ref{lem:722}) and \cite[Lemma 7.2]{Y3}, the right hand side is nonzero at $s=1$ and holomorphic for $\Re s>0$. 
Thus $L_\bff(s,\Ik_n(\sig))$ has a pole at $s=1$ and is holomorphic for $\Re s>1$. 
Since $\frkD_{(2\kap_v+n)/2}^{(n)}$ is tempered, its $L$-factor is holomorphic for $\Re s>0$ by Lemma 7.2 of \cite{Y3}. 
Therefore the complete $L$-function of $\Ik_n(\sig)$ has a pole at $s=1$ and is holomorphic for $\Re s>1$. 

Proposition \ref{prop:62}(\ref{prop:621}) gives a quadratic space $\calv$ of dimension $2n$ and discriminant $1$ such that $\tht^\psi_\calv(\Ik_n(\sig))$ is nonzero and cuspidal. 
Because of the choice of $\eta$, Proposition \ref{prop:81} shows that
\begin{align*}
\calv&\simeq \calv_n^\vep, & 
\tht^\psi_{\calv_n^\vep}(\Ik_n(\sig))&\simeq\Xi_n^{\vep}(\sig)\simeq\Xi_n^\vep(\pi)\otimes\hat\chi^\eta\circ x^\vep_n. 
\end{align*}
In particular, $\rmm_\cusp(\Xi_n^\vep(\pi))\geq 1$. 
Arguing as above, one can see that the standard $L$-function of $\Xi_n^{\vep}(\sig)$ is entire and has no zero at $s=1$. 
We get  
\[\rmm_\cusp(\Xi_n^\vep(\pi))=\rmm_\cusp(\Xi_n^\vep(\sig))\leq\rmm_\cusp(\tht^\psi_n(\Xi_n^{\vep}(\sig)))=\rmm_\cusp(\Ik_n(\sig))=1\]
by Proposition \ref{prop:62}(\ref{prop:622}).  
\end{proof}

\begin{theorem}\label{thm:92}
Let $\pi\simeq\otimes_v'\pi_v$ be an irreducible cuspial automorphic representation of $\PGL_2(\AA)$ such that $\pi_v\simeq D_{2\kap_v}$ with $\kap_v\geq\frac{n}{2}$ for $v\in\frkS_\infty$. 
If none of $\pi_\frkp$ is supercuspidal and $L\bigl(\frac{1}{2},\pi\bigl)\neq 0$, then the multiplicity of irreducible summands of $\rmSp_n^\eps(\pi)$ in $\scra_\cusp(Sp_n)$ is $\frac{1}{2}(1+(-1)^{dn/2}\prod_\frkp\eps_\frkp)$. 
\end{theorem}

\begin{proof}
Let $\vPi$ be an irreducible summand of $\rmSp_n^\eps(\pi)$. 
Suppose that $\vPi$ is a cuspidal automorphic representation of $Sp_n(\AA)$. 
Since $Sp_n^+(\pi_v)^g\simeq Sp_n^+(\pi_v)$ for all $g\in\GSp_n(F_v)$ and since $\frkD^{(n)}_k{}^g\simeq\bar\frkD^{(n)}_k$ if $g\in\GSp_n(\RR)$ with $\lam_n(g)<0$, we may assume that $\vPi_v\simeq\frkD^{(n)}_{(2\kap_v+n)/2}$ for $v\in\frkS_\infty$. 
In the proof of Theorem \ref{thm:92} we have observed that $L(s,\vPi)$  is holomorphic for $\Re s>1$ and has a pole at $s=1$. 
Proposition \ref{prop:62}(\ref{prop:621}) gives a quadratic space $\calv$ of dimension $2n$ and discriminant $1$ such that $\tht^\psi_\calv(\vPi)$ is nonzero. 
Since $\calv(F_v)\simeq\calv_n^{\eps_v}$ for all $v$ by Proposition \ref{prop:81}, we have $\prod_v\eps_v=1$. 
Since $\tht^\psi_\calv(\vPi)$ is equivalent to an irreducible summand of $\rmO_n^\eps(\pi)$ by Lemma \ref{lem:81}, the multiplicity preservation in Proposition \ref{prop:62}(\ref{prop:621}) and multiplicity one in Theorem \ref{thm:91} show that 
\[\rmm_\cusp(\vPi)\leq\rmm_\cusp(\tht^\psi_\calv(\vPi))=1. \] 

Suppose that $\prod_v\eps_v=1$. 
By Theorem \ref{thm:91} $\rmO^\eps_n(\pi)$ is a cuspidal automorphic representation of $\O(\calv_n^\eps,\AA)$. 
When $\sig$ is an irreducible summand of $\rmO^\eps_n(\pi)$, we can show that $L(s,\sig)$ is entire and has no zero at $s=1$ as in the proof of Theorem \ref{thm:91}. 
Proposition \ref{prop:62}(\ref{prop:622}) and Lemma \ref{lem:81} show that 
\[(\otimes_{v\in\frkS_\infty}\frkD^{(n)}_{(2\kap_v+n)/2})\otimes(\otimes'_\frkp \rmSp^{\eps_\frkp}_n(\pi_\frkp))\simeq\tht^\psi_n(\rmO^\eps_n(\pi))\] 
is a cuspidal automorphic representation of $Sp_n(\AA)$. 
Thus $\rmm_\cusp(\vPi)\geq 1$. 
\end{proof}

\section{Duke-Imamoglu-Ikeda lifts and theta functions}\label{sec:10}
 
We denote the integer ring of $F$ by $\frko$, that of $F_\frkp$ by $\frko_\frkp$ and the different of $F/\QQ$ by $\frkd=\prod_\frkp\frkd_\frkp$. 
From now on we assume that $dn$ is a multiple of $4$. 
Then there exists a totally positive definite quadratic space $(\calv_n,(\;,\;))$ of rank $2n$ that is split at every finite prime. 
To simplify the matter, we suppose that $\pi$ is an irreducible cuspidal automorphic representation of $\PGL_2(\AA)$ whose archimedean component is $D_n^{\boxtimes d}$ with even natural number $n$ and whose nonarchimedean component is an irreducible principal series $\otimes'_\frkp I(\mu_\frkp^{},\mu_\frkp^{-1})$. 
For ease of notations we put 
\begin{align*}
Sp_n(\pi)&=\frkD_n^{(n)\boxtimes d}\otimes(\otimes'_\frkp I_n(\mu_\frkp))
\simeq\Ik_n(\pi), \\
\O_n(\pi)&=\1_{\O(\calv_n,\AA_\infty)}\otimes(\otimes'_\frkp J_n(\mu_\frkp))\simeq\O_n^{\eps_n(\pi)}(\pi). 
\end{align*}
(cf. (\ref{tag:81}), Lemmas \ref{lem:74}(\ref{lem:742}) and \ref{lem:75}(\ref{lem:755}))
If $\vep\bigl(\frac{1}{2},\pi\bigl)=1$, then Proposition \ref{prop:91} and Theorem \ref{thm:91} give rise to intertwining embeddings 
\begin{align*}
i_n&:Sp_n(\pi)\hookrightarrow\scra_\cusp(Sp_n), & 
j_n&:\O_n(\pi)\hookrightarrow\scra_\cusp(\O(\calv_n)). 
\end{align*} 
These intertwining maps are unique up to scalar multiple. 

The Weil representation $\ome^\psi_{\calv_n}$ can be realized on the space $\cals(\calv_n^n(\AA))$ of Schwartz functions on $\calv_n^n(\AA)$ with 
\begin{align*}
\left[\ome^\psi_{\calv_n}\left(\begin{pmatrix} a & b\trs a^{-1} \\ 0 & \trs a^{-1}\end{pmatrix}\right)\phi\right](x)&=\psi(bQ(x,x))\alp(a)^n\phi(xa), \\ 
[\ome^\psi_{\calv_n}(h)\phi](x)&=\phi(h^{-1}x)
\end{align*}
for $\phi\in\cals(\calv_n^n(\AA))$, $a\in\GL_n(\AA)$, $b\in\Sym_n(\AA)$ and $h\in\O(\calv_n,\AA)$, where $Q(x)=\frac{1}{2}((x_i,x_j))\in\Sym_n(\AA)$ is the matrix of inner product of the components of $x=(x_1,\dots,x_n)\in\calv_n^n(\AA)$. 
Let $S(\calv_n^n(\AA))$ be the subspace of $\cals(\calv_n^n(\AA))$ which corresponds to a polynomial Fock space at every real prime. 

We shall construct a nonzero $Sp_n(\AA)\times\O(\calv_n,\AA)$-intertwining map 
\[\vth_n^\psi=\otimes_v\vth^{\psi_v}_n:\ome^\psi_{\calv_n}\otimes \O_n(\pi)\twoheadrightarrow Sp_n(\pi) \] 
explicitly. 
If $v\in\frkS_\infty$, then since $\Tht^{\psi_v}_n(\1_{\O(\calv_n,F_v)})$ is irreducible by Theorem 4.6 of \cite{LZ}, we can define by Rallis's theorem the $\O(\calv_n,F_v)$-invariant $Sp_n(F_v)$-intertwining map by 
\begin{align*}
\vth_n^{\psi_v}&:\ome^{\psi_v}_{\calv_n}\twoheadrightarrow \frkD_n^{(n)}\hookrightarrow I_n(\alp_v^{(n-1)/2}), &
\vth_n^{\psi_v}(g,f_v)&=[\ome^{\psi_v}_{\calv_n}(g)f_v](0). 
\end{align*}

We call a lattice $L$ in a quadratic space $\calv$ over $F$ even if $(x,x)\in 2\frko$ for every $x\in L$, and unimodular if $L$ equals its dual lattice $\{x\in\calv(F)\;|\;(x,L)\in\frko\}$. 
As is well-known, even unimodular lattices form a single genus. 
For $h\in\O(\calv,\AA)$, we write $hL$ for the lattice defined by $(hL)_\frkp=h_\frkp L_\frkp$, where we denote the closure of $L$ in $\calv(F_\frkp)$ by $L_\frkp$. 
Put 
\begin{align*}
\calk_L&=\{h\in\O(\calv,\AA)\;|\;hL=L\}, & 
\O(L)&=\O(\calv,F)\cap\calk_L, &
E(L)&=\sharp\O(L). 
\end{align*}
Fix an even unimodular lattice $\scrl$ in $\calv_n(F)$. 
Let $\Xi_n$ denote a fixed set of lattices that are representatives for the classes belonging to the genus of $\scrl$. 
We identify the finite set $\O(\calv_n,F)\bsl \O(\calv_n,\AA)/\calk_\scrl$ with $\Xi_n$ via the map $h\mapsto h\scrl$ and regard $\calk_\scrl$-invariant automorphic forms on $\O(\calv_n,\AA)$ as functions on the set $\Xi_n$. 
Its standard $L$-function is defined in \cite{PSR,Y3}. 

Fix an $\frko_\frkp$-basis $e_{\frkp,1},\dots,e_{\frkp,n},f_{\frkp,1},\dots,f_{\frkp,n}$ of $L_\frkp$ such that 
\begin{align*}
(e_{\frkp,i},e_{\frkp,j})&=0, & 
(f_{\frkp,i},f_{\frkp,j})&=0, & 
(e_{\frkp,i},f_{\frkp,j})&=\del_{ij}.  
\end{align*}
Put $e_\frkp=(e_{\frkp,1},e_{\frkp,2},\dots,e_{\frkp,n})\in\calv_n^n(F_\frkp)$. 
Following \cite{LTZ}, we consider the following integral, which is absolutely convergent for $\Re s\geq\frac{n}{2}$,  
\[Z(\phi_\frkp,s,\mu_\frkp)=\int_{\GL_n(F_\frkp)}\mu_\frkp(\det a)^{-1}\phi_\frkp(e_\frkp a)\alp_\frkp(\det a)^{s+(n-1)/2}\d a\]
for $\phi_\frkp\in\cals(\calv_n^n(F_\frkp))$.  
Put 
\[Z(\phi_\frkp,\mu_\frkp)=\lim_{s\to n/2}L^\mathrm{GJ}(s,\mu_\frkp^{-1}\circ{\det}_{\GL_n})^{-1}Z(\phi_\frkp,s,\mu_\frkp), \]
where $L^\mathrm{GJ}(s,\mu_\frkp^{-1}\circ{\det}_{\GL_n})=\prod_{j=1}^nL\bigl(s+\frac{n+1}{2}-j,\mu_\frkp^{-1}\bigl)$ is the Godement-Jacquet $L$-factor of the one-dimensional representation $\mu_\frkp\circ\det$ of $\GL_n(F_\frkp)$. 
A simple calculation gives 
\begin{multline*}
Z\left(\ome^{\psi_\frkp}_{\calv_n}\left(\begin{pmatrix} a_1 & b_1\trs a_1^{-1} \\ 0 & \trs a_1^{-1}\end{pmatrix},\begin{pmatrix} a_2 & b_2\trs a_2^{-1} \\ 0 & \trs a_2^{-1}\end{pmatrix}\right)\phi_\frkp,\mu_\frkp\right)\\
=\mu_\frkp(\det a_1)\alp_\frkp(\det a_1)^{n+1/2}\mu_\frkp(\det a_2)^{-1}\alp_\frkp(\det a_2)^{n-1/2}Z(\phi_\frkp,\mu_\frkp). 
\end{multline*} 
We can define the intertwining map $\vth_n^{\psi_\frkp}:\ome^{\psi_\frkp}_{\calv_n}\otimes J_n(\mu_\frkp)\twoheadrightarrow I_n(\mu_\frkp)$ by setting 
\[\vth_n^{\psi_\frkp}(g,\phi_\frkp\otimes f_\frkp)=\int_{Q^n_n(F_\frkp)\bsl\O(\calv_n,F_\frkp)}Z(\ome^{\psi_\frkp}_{\calv_n}(g,h)\phi_\frkp,\mu_\frkp)f_\frkp(h)\,\d h. \]

\begin{theorem}\label{thm:101}
Let $\pi\simeq D_n^{\boxtimes d}\otimes(\otimes'_\frkp I(\mu_\frkp^{},\mu_\frkp^{-1}))$ be an irreducible cuspidal automorphic representation of $\PGL_2(\AA)$. 
If $L\bigl(\frac{1}{2},\pi\bigl)\neq 0$, then there is a nonzero constant $c$ such that for all $\phi\in S(\calv_n^n(\AA))$ and $f\in\O_n(\pi)$
\[i_n(g,\vth^\psi_n(\phi\otimes f))=c\int_{\O(\calv_n,F)\bsl\O(\calv_n,\AA)}\overline{j_n(h,f)}\Tht(\ome^\psi_{\calv_n}(g,h)\phi)\,\d h. \]
\end{theorem}

\begin{proof}
As was seen in the proof of Theorem \ref{thm:92}, the theta lift of $\O_n(\pi)$ is equal to $Sp_n(\pi)$. 
The Howe principle forces the map $\phi\otimes f\mapsto\tht^\psi_n(j_n(f),\phi)$ to factor through the map $\vth^\psi_n$. 
It is proportional to $i_n\circ\vth^\psi_n$ by uniqueness. 
\end{proof}

Finally, we translate our results into a more classical language. 
Let $\AA_\infty=F\otimes_\QQ\RR$ be the infinite part of the ad\`{e}le ring. 
The Siegel upper half space $\calh_m$ of degree $m$ consists of all complex symmetric matrices of size $m$ with positive definite imaginary part. 
The subset $\Sym_m^+$ of $\Sym_m(F)$ consists of totally positive definite symmetric matrices. 
For fractional ideals $\frkb,\frkc$ of $\frko$ we put 
\beq
\Gam_m[\frkb,\frkc]=\left\{\begin{pmatrix} A & B \\ C & D \end{pmatrix}\in Sp_m(F)\;\biggl|\; \begin{matrix} A,D\in\Mat_m(\frko) \\ 
B\in\frkb\Mat_m(\frko),\;C\in\frkc\Mat_m(\frko)\end{matrix}\right\}. \label{tag:101} 
\eeq

We define the action of $Sp_m(\RR)$ on $\calh_m$ and the automorphy factor by 
\begin{align*}
g\calz&=(A\calz+B)(C\calz+D)^{-1}, & 
J(g,\calz)&=\det(C\calz+D)
\end{align*}
for $\calz\in\calh_m$ and $g=\begin{pmatrix} A & B \\ C & D \end{pmatrix}\in Sp_m(\RR)$. 
For $\kap\in\ZZ$, $g=(g_v)\in Sp_m(\AA_\infty)$ and a $\CC$-valued function $\calf$ on $\calh_m$ we define a function $\calf|_\kap g:\calh_m^d\to\CC$ by 
\[\calf|_\kap g(\calz)=\calf(g\calz)J(g,\calz)^{-\kap}, \]
where $g\calz=(g_v\calz_v)$ and $J(g,\calz)=\prod_vJ(g_v,\calz_v)$. 

Define the character $\bfe_\infty:\CC^d\to\CC^\times$ by $\bfe_\infty(\calz)=\prod_{v\in\frkS_\infty}e^{2\pi\iu\calz_v}$ for $\calz=(\calz_v)\in\CC^d$. 
A holomorphic function $\calf$ on $\frkH^d_m$ is called a Hilbert-Siegel cusp form of weight $\kap$ with respect to $\Gam_m[\frkb,\frkc]$ if $\calf|_\kap\gam=\calf$ for all $\gam\in\Gam_m[\frkb,\frkc]$ and for all $\bet\in Sp_m(F)$, $\calf|_\kap\bet$ has a Fourier expansion of the form 
\[\calf|_\kap\bet(\calz)=\sum_{\xi\in \Sym^+_m}c(\xi)\bfe_\infty(\tr(\xi\calz)). \] 
The theta series associated to $L\in\Xi_i$ is a Hilbert-Siegel modular form of degree $j$ and weight $i$ defined by 
\[\tht_L^{(j)}(\calz)=\sum_{x\in L^j}\bfe_\infty(\tr(Q(x)\calz))=\sum_{\xi\in\Sym_j(F)}N(\xi,L)\bfe_\infty(\tr(\xi\calz)).  \] 

Let $\pi\simeq D_{2k}^{\boxtimes d}\otimes(\otimes'_\frkp I(\alp_\frkp^{s_\frkp},\alp_\frkp^{-s_\frkp}))$ be an irreducible cuspidal automorphic representation of $\PGL_2(\AA)$. 
Unless otherwise indicated, all terminology and notations will follow those of Section \ref{sec:1}. 
Suppose that $dk$ is even. 
Let 
\[h_m(\calz)=\sum_{\eta\in\Sym_1^+} c(\eta)\bfe_\infty(\eta \calz)\frkf_{(-1)^m\eta}^{(2k-1)/2}\prod_\frkp\Psi_\frkp((-1)^m\eta,q_\frkp^{s_\frkp})\in S^+_{(2k+1)/2} \]
be an element which generates $\otimes'_v\Mp_1^{\psi_v}(\sig_v^{\eps_1(\sig_v)})$, where $\sig=\pi\otimes\hat\chi^{(-1)^m}$. 
Corollary 10.1 of \cite{IY} says that the series 
\[\calf_m(Z)=\sum_{\xi\in\scrr_{2m}\cap\Sym_{2m}^+} c(\det(2\xi))\frkf_\xi^{(2k-1)/2}\prod_\frkp\wtl{F}_\frkp(\xi,q_\frkp^{s_\frkp})\bfe_\infty(\tr(\xi\calz)) \] 
is a Hilbert-Siegel cusp form of weight $k+m$ with respect to  $\Gam_{2m}[\frkd^{-1},\frkd]$ which generates the representation $\Ik_{2m}(\pi)$ defined in (\ref{tag:91}). 

Theorem \ref{thm:101} can be made explicit with a suitable choice of test functions. 
\begin{corollary}\label{cor:101} 
Notation being as above, we let $m=k$. 
If $f$ is a nonzero $\calk_\scrl$-invariant vector in $\O_{2k}(\pi)$ and $L\bigl(\frac{1}{2},\pi\bigl)\neq 0$, then the sum 
\beq
\sum_{L\in\Xi_{2k}} f(L)\frac{\tht^{(2k)}_L(\calz)}{E(L)} \label{tag:102}
\eeq
is a nonzero Hilbert-Siegel cusp form of weight $n$ with respect to $\Gam_{2k}[\frkd^{-1},\frkd]$ and  (\ref{tag:102}) is equal to $\calf_k$ up to scalar. 
\end{corollary}

\begin{proof}
We define $\phi_\scrl=(\otimes_{v\in\frkS_\infty}\phi_v^0)\otimes(\otimes_\frkp\phi_{\scrl_\frkp})\in S(\calv_{2k}^{2k}(\AA))$ by $\phi^0_v(x)=e^{-2\pi\tr(Q(x))}$ for $x\in\calv_{2k}^{2k}(F_v)$ and by letting $\phi_{\scrl_\frkp}$ be the characteristic function of $\scrl_\frkp^{2k}$. 
Observe that $\phi_{\scrl_\frkp}$ is invariant under the action of the product $\Gam_{2k}[\frkd_\frkp]\times\calk_{\scrl_\frkp}$ of maximal compact subgroups, where  
\begin{align*}
\Gam_{2k}[\frkd_\frkp]&=\left\{\begin{pmatrix} A & B \\ C & D \end{pmatrix}\in Sp_{2k}(F_\frkp)\;\biggl|\; \begin{matrix} A,D\in\Mat_{2k}(\frko_\frkp) \\ 
B\in\frkd_\frkp^{-1}\Mat_{2k}(\frko_\frkp),\;C\in\frkd_\frkp\Mat_{2k}(\frko_\frkp)\end{matrix}\right\}, \\
\calk_{\scrl_\frkp}&=\{h\in\O(\calv_{2k},F_\frkp)\;|\;h\scrl_\frkp=\scrl_\frkp\}.   
\end{align*}
On the other hand, for $v\in\frkS_\infty$, $g\in Sp_{2k}(F_v)$, $h\in\O(\calv_{2k},F_v)$, $x\in\calv_{2k}(F_v)$
\[[\ome^{\psi_v}_{\calv_{2k}}(g,h)\phi^0_v](x)=J(g,\iu\ono_{2k})^{-2k}e^{-2\pi\tr(Q(x)g(\iu{\bf1}_{2k}))}. \]
Put $\bfi_{2k}=(\iu\ono_{2k},\dots,\iu\ono_{2k})\in\calh_{2k}^d$. 
Then 
\begin{align*}
\Tht(\ome^\psi_{\calv_{2k}}(g,h)\phi_\scrl)&=\tht_{h\scrl}(g(\bfi_{2k}))/J_{2k}(g,\bfi_{2k})
&(g&\in Sp_{2k}(\AA_\infty),\; h\in\O(\calv_{2k},\AA)). 
\end{align*}
If we put $\calz=g(\bfi_{2k})$, then $J_{2k}(g,\bfi_{2k})\tht^\psi_{\calv_{2k}}(g,f\otimes\phi_\scrl)$ equals the sum (\ref{tag:102}). 

If we write $f=j_{2k}(\otimes_\frkp^{} f^0_\frkp)$ with $\calk_{\scrl_\frkp}$-invariant vector $f^0_\frkp\in J_{2k}(\alp_\frkp^{s_\frkp})$, then one can readily see that $\vth^{\psi_\frkp}_{2k}(\phi_{\scrl_\frkp}\otimes f^0_\frkp)$ is a nonzero $\Gam_{2k}[\frkd_\frkp]$-invariant vector in $I_{2k}(\alp_\frkp^{s_\frkp})$. 
Since $\vth^{\psi_v}_{2k}(\phi^0_v)$ is a nonzero lowest weight vector in $\frkD^{(2k)}_{2k}$, Proposition 8.9 of \cite{IY} shows that $i_{2k}^{}(\vth^\psi_{2k}(\phi_\scrl\otimes(\otimes^{}_\frkp f^0_\frkp)))$ is a nonzero Hilbert-Siegel cusp form of weight $2k$ with respect to $\Gam_{2k}[\frkd^{-1},\frkd]$. 
Since $\rmm_\cusp(\Ik_{2k}(\pi))=1$, the series $\calf_k$ and (\ref{tag:102}) are equal up to scalar.  
\end{proof}

\begin{remark}\label{rem:101}
\begin{enumerate}
\renewcommand\labelenumi{(\theenumi)}
\item\label{rem:1011} When $k=1$ and $F=\QQ$, the theta lift of this type has been studied by Yoshida \cite{Yo} (cf. \cite[Theorem 5.4]{Sc}). 
\item\label{rem:1012} The linear subspace of $M_{12}(Sp_{12}(\ZZ))$ spanned by the theta series associated to the 24 different Niemeier lattices is extensively studied in \cite{BFW,NV,I2,I6,C}. 
It intersects the space of cusp forms in a one dimensional subspace. 
Its basis vector is explicitly given in Theorem 4 of \cite{BFW}.  
Letting $F=\QQ$ and $k=6$, Ikeda verified that this subspace is spanned also by his lift $\calf_6$ of the Ramanujan delta function 
in Section 15 of \cite{I2}.  
\item\label{rem:1013} The function $f$ can be obtained by the integral
\[f(L)=\int_{\Gam_{2k}^{}[\frkd^{-1},\frkd]\bsl\calh_{2k}^d}\calf_k(\calz)\overline{\tht^{(2k)}_L(\calz)}\,\prod_{v\in\frkS_\infty}(\Im\calz)^{2k}\d\mu, \]
where $\d\mu$ denotes the $Sp_{2k}(\AA_\infty)$-invariant top degree differential form on $\calh_{2k}^d$. 
When $F=\QQ$, we can view our identity as a refinement of the general formula (\Roman{nin}.1) of \cite{B}.   
\end{enumerate}
\end{remark}


\end{document}